\documentclass[12pt,a4paper]{article}
\usepackage{amsmath,amssymb,amsthm}
\usepackage{hyperref}
\usepackage{enumitem}
\usepackage{mathtools}
\usepackage{xcolor}

\newtheorem{theorem}{Theorem}[section]
\newtheorem{lemma}[theorem]{Lemma}

\theoremstyle{definition}

\newtheorem{remark}[theorem]{Remark}
\newtheorem{example}[theorem]{Example}
\newtheorem{assum}{Assumption}[section]

\newcommand{\R}{\mathbb{R}}

\numberwithin{equation}{section}
\hypersetup{colorlinks=true, linkcolor=blue, citecolor=blue}

    \makeatletter
\def\namedlabel#1#2{\begingroup
	#2%
	\def\@currentlabel{#2}%
	\phantomsection\label{#1}\endgroup
}
\makeatother

\allowdisplaybreaks

\begin{document}

\title{Slow and fast dynamics in measure functional differential equations with state--dependent delays through averaging principles and applications to
extremum seeking}

\author{Jaqueline G. Mesquita$^{a}$, \ Tiago Roux Oliveira$^{b}$, \ Henrique C. dos Reis$^{a}$ \\ \\
	{\small{ $^{a}$Universidade de Bras\'{\i}lia, Departamento de Matem\'atica, Campus Universit\'ario}}\\ 
	{\small{ Darcy Ribeiro, Asa Norte 709109--900, Bras\'{\i}lia--DF, Brazil.}}\\ 
	{\small{\textcolor{black}{E-mail: jgmesquita@unb.br, henrique.costa.reis@hotmail.com}}}\ \\ \\
	{\small{\textcolor{black}{$^{b}$Universidade do Estado do Rio de Janeiro, Faculdade de Engenharia,}}}\\ 
	{\small{\textcolor{black}{20550--900, Rio de Janeiro--RJ, Brazil.}}}\\
	{\small{\textcolor{black}{E-mail: tiagoroux@uerj.br}}}
	}

\date{}

\maketitle

\begin{abstract}
This paper investigates a new class of equations called measure functional differential
equations with state-dependent delays. We establish the existence and
uniqueness of solutions and present a discussion concerning the appropriate phase
space to define these equations. Also, we prove a version of periodic averaging principle
to these equations. This type of result was completely open in the literature.
These equations involving measure bring the advantage to encompass others such
as impulsive, dynamic equations on time scales and difference equations, expanding
their application potential. Additionally, we apply our theoretical insights to a
real-time optimization strategy, using extremum seeking to validate the stability of
an innovative algorithm under state-dependent delays. This application confirm the
relevance of our findings in practical scenarios, offering valuable tools for advanced
control system design. Our research provides significant contributions to the mathematical
field and suggests new directions for future technological developments.
\end{abstract}

\section{Introduction}

Since the 1960s (\cite{38}), it has been well established that many real-world systems can
be effectively modeled by retarded functional differential equations with state-dependent
delays. For numerous examples and models that are better described by such equations,
we refer the reader to the comprehensive Chapter by Hartung et al. (\cite{20}). This growing
recognition has led to an increasing focus on the theory of functional differential equations (FDEs) 
with state-dependent delays, attracting significant attention from researchers over the past
few decades. Key contributions to the foundational aspects of this theory can be found
in \cite{4, 5, 12, 19, 49}, among others.

More recently, there has been progress in extending this theory to encompass partial
differential equations and integro-differential equations with state-dependent delays. For
an overview of the latest developments in this area, we direct the reader to \cite{1, 2, 3, 5, 21,
22, 23, 37} and the references therein. Additionally, there has been considerable interest in
the development of stability and control theory for systems governed by state-dependent
retarded differential equations. Recent advances in this field can be found in works such
as \cite{6, 46, 48, 50}. However, the framework to deal with these equations always consider
space of piecewise continuous functions, continuous functions or even $C^{1}$ functions, not
considering equations with regulated functions, which extend all the previous one. This
necessity comes from the fact to investigate measure type of equations, which is our main
goal here.

More precisely, we are interested here in the following types of equations involving
measure
 \begin{equation}
 	\begin{array}{rcl}
 		\label{eq1.1_intro}
 		x(t) &=& x(t_0)+ \displaystyle\int_{t_0}^{t} f(s,x_{\rho(s,x_s)}){\rm d}g(s), \quad t \in [t_0,t_0+\sigma], \\ x_{t_0}&=&\phi,
 	\end{array}
 \end{equation}
 where $\sigma>0$, $t_0 \in \mathbb{R}$, $g\colon [t_0,t_0+\sigma] \to \mathbb{R}$ is a  nondecreasing function, $\mathcal{B} \subset G((-\infty,0],\mathbb{R}^n)$ is a particular Banach space satisfying axioms \ref{a1}--\ref{A3}, which will be explained later, $\phi \in \mathcal{B}$ and $x\colon (-\infty,t_0+\sigma] \to \mathbb{R}^n$, $\rho\colon [t_0,t_0+\sigma]\times \mathcal{B} \to \mathbb{R}$ and 
 $f\colon  [t_0,t_0+\sigma]\times \mathcal{B} \to \mathbb{R}^n$ are functions.

These equations introduce a significant novelty in the literature, as measure equations
have not yet been explored in the context of state-dependent delays, being for the first time
in the present paper. The main challenge stems from the functional spaces in which the
involved functions are defined. Discontinuities in these functions complicate the analysis,
particularly when composing two such functions. For instance, in this paper, we work with
regulated functions, but the composition of two regulated functions does not necessarily
yield a regulated function. Consequently, defining solutions in this space requires careful
investigation of an appropriate phase space to ensure the well-posedness and smooth
operation of the solution process.

Some of the earliest advances to the theory of measure differential equations (MDEs)
are due to W. W. Schmaedeke, R. R. Sharma, and P. C. Das (see \cite{11, 44, 45}). Over time,
several researchers have significantly advanced the qualitative theory of these equations,
with a focus on existence of solutions, stability theory, and their applications to systems
exhibiting discontinuous behavior. Notable contributions in this area include \cite{9, 13}. A
key motivation for studying these equations lies in the fact that these equations may encompass
other types of equations such as: dynamic equations on time scales and impulsive
equations. The same can be viewed in the case of state–dependent delays, bringing much
more generality for the results, since they can be rewritten for these equations as special
cases, in the context of state–dependent delays.

On the other hand, averaging principle play an important role for simplifying the analysis
of nonautonomous differential systems by approaching them, under certain assumptions,
into autonomous systems, reducing time-varying perturbations to time-invariant
ones with minimal error.

In this paper, our goal is to investigate a version of periodic averaging principle for
measure functional differential equations with state–dependent delays. Important results
on averaging principle on functional equations can be found in \cite{17, 18, 29, 30, 31, 32}.

Despite the well-established nature of averaging for various equations, its application
to functional differential equations with state-dependent delays has remained challenging
due to the regularity of the system and the complexity of the spaces involved. This paper
aims to introduce a result on existence and uniqueness of solutions of equation (\ref{eq1.1_intro}) and
to prove periodic averaging methods to measure functional differential equations with
state-dependent delays. More precisely, we ensure that, under certain conditions, the
solutions of the system below
	\begin{equation}\label{eq-1_intro}
\left\{
	\begin{array}{rcl}
	x(t) &=& \displaystyle x(0) +\varepsilon\int_{0}^{t} f\left(s, x_{\rho(s,x_s,\varepsilon)}\right){\rm d}h(s) + \varepsilon^2\int_{0}^{t}g\left(s, x_{\rho(s,x_s,\varepsilon)},\varepsilon \right){\rm d}h(s), \\ x_{0}&=& \phi,
	\end{array}
\right.
	\end{equation}
may be approached by the solutions of the system below
\begin{equation}\label{eq-2_intro}
\left\{
	\begin{array}{rcl}
	y(t) &=& \displaystyle y(0) +\varepsilon\int_{0}^{t} f_0\left(y_{\rho(s,y_s, \varepsilon)}\right){\rm d}s \\ y_{0}&=& \phi,
	\end{array}
\right.
	\end{equation}
which is simpler to deal, and provides a tool to understand the asymptotic behavior of
the solution of (\ref{eq-1_intro}).	


Our work also demonstrates the effectiveness of periodic averaging in simplifying these
equations, enabling stability analysis, bifurcation studies, and synchronization phenomena
in delayed systems. This method is particularly powerful for developing control strategies
in diverse fields such as neuroscience, communication networks, and ecological modeling.
On the other hand, the integration of the Perron integral, considered here in this paper,
particularly in handling non-standard integrands in averaged models, such as the Kapitza
pendulum, offers further refinement in the analysis of dynamic systems and for this type
of equations. This method has proven particularly useful in addressing the challenges
posed by state-dependent delays and non-periodic perturbations in complex systems.

Applying the averaging result, we present an extremum seeking algorithm \cite{AOH:2021,39} for real–
time optimization of static maps under nonconstant state–dependent delays and prove its
convergence. This kind of result was completely open in the literature in the context of
state–dependent delays until now \cite{OK:2015_G,OK:2015_N,40}.

In summary, this paper advances the theoretical framework of averaging methods in
differential equations, providing novel extensions and addressing challenges. The results
have broad implications across various scientific fields, from optimizing renewable energy
systems to understanding dynamic responses in electrodynamics and mathematical
biology.

\textbf{Preliminaries.}~~ The main references for this subsection are \cite{14,43}.


A {\em tagged division} of  $\left[ a,b\right]
$ is a finite collection of point--interval pairs
$D=\left( \tau _{i},[s_{i-1},s_{i}]\right) $, where $a=s_{0}\leq
s_{1}\leq \ldots \leq s_{\left|D\right|}=b$ is a division of $\left[a,b\right]$
and $\tau _{i}\in \left[ s_{i-1},s_{i}\right]$, $i=1,2,\ldots,\left|D\right|$, where the symbol $\left|D\right|$ denotes the number of subintervals in which $[a,b]$ is divided.

A {\em gauge} on a set $B\subset \left[ a,b\right] $ is any
function $\delta :B\rightarrow (0,\infty) $. Given a gauge
$\delta $ on $\left[ a,b\right] $, we say that a tagged division $D=\left(
\tau_{i},[s_{i-1},s_{i}]\right)$  is $\delta${\em--fine} if for
every $i\in\left\{1,2,\ldots,\left|D\right|\right\}$, we have
$\left[ s_{i-1},s_{i}\right] \subset (\tau_{i}-\delta(\tau_{i}),\tau_{i}+\delta(\tau_i)).$

A function $f \colon [a,b] \to X$ is called {\em Perron--Stieltjes} integrable on $[a,b]$ with respect to a function $g: [a,b] \to \R,$ if there is an element $I\in X$ such that for every $\varepsilon>0,$ there is a gauge $\delta \colon [a,b]\rightarrow (0,\infty)$  such that
\[\left\|\sum\limits_{i=1}^{\left|D\right|}f(\tau_i)\left(g(s_i)-g(s_{i-1})\right) - I\right\|<\varepsilon,\]for all $\delta$--fine tagged partition of $[a,b].$ In this case, $I$ is called {\em Perron--Stieltjes} integral of $f$ with respect to $g$ over $[a,b]$ and it will be denoted by $\int^b_a f(s)\,{\rm d}g(s),$ or simply $\int^b_a f\,{\rm d}g.$

A function $f \colon [a,b]\to X$ is called \textit{regulated} if both lateral limits $$f(t^-) = \displaystyle
\lim_{s \to t^-} f(s),\ \ \  t\in (a,b] \qquad \text{and} \qquad f(t^+) = \displaystyle
\lim_{s \to t^+} f(s),\ \ \ t\in[a,b)$$exist. The space of all regulated functions $f: [a,b]\to X$ will be denoted by $G([a,b],\R^n),$ which is a Banach space when endowed with the usual supremum norm  $\| f \|_{\infty}=\sup\limits_{s\in[a,b]}\left|f(s)\right|.$

A set $\mathcal{A}\subset G([a,b],\R^n)$ is called \textit{equiregulated}, if it has the following property: for every $\varepsilon >0$ and $t_0\in [a,b],$  there is a $\delta>0$ such that
\begin{enumerate}[label=(\roman*)]
	\item if $x\in\mathcal{A},$ $s\in [a,b]$ and $t_0 -\delta <s< t_0,$ then $\|x(t^{-}_0)-x(s)\|<\varepsilon$, whereas 
	
	\item if $x\in\mathcal{A},$ $s\in [a,b]$ and $t_0 <s< t_0 +\delta,$ then $\|x(t^{+}_0)-x(s)\|<\varepsilon.$
\end{enumerate}

\begin{theorem}[\protect{\cite[Theorem 2.18]{14}}]\label{relativamente-compacto}
	The following conditions are equivalent.
	\begin{itemize}
		\item[\rm{(i)}] $\mathcal{A}\subset G([a,b],\R^n)$ is relatively compact.
		
		\item[\rm{(ii)}]  The set $\{x(a):x\in\mathcal{A}\}$ is bounded and there is an increasing continuous function $\eta: [0, \infty) \to [0, \infty)$, $\eta(0) = 0$ and there is an increasing function $K:[a,b] \to \R$ such that
		\[\|x(\tau_2)-x(\tau_1)\|\leq \eta(K(\tau_2)- K(\tau_1)),\]for all $x\in \mathcal{A}$ and all $a\leq \tau_1\leq \tau_2\leq b.$
	\end{itemize}
\end{theorem}

\begin{theorem}[\cite{43}, Gronwall Inequality for Perron-Stieltjes]\label{Gronwall}
	Let $g\colon [a,b] \to [0,\infty)$ be a nondecreasing and left--continuous function, $k\geq0$ and $l> 0.$ Assume that $\psi\colon[a,b] \to [0,\infty)$ satisfies
	\[\psi(\xi)\leq k +l\int_a^{\xi} \psi (s)\mathrm{d}g(s),\quad \xi\in[a,b].\]Then $\psi(\xi)\leq k e^{l(g(\xi)-g(a))}$ for all $\xi\in[a,b].$
\end{theorem}

\begin{theorem}[Schauder  Fixed--Point  Theorem]\label{punto-fijo}
	Let $(E, \|  \cdot  \|)$ be a normed vector space, $S$ be a  nonempty  convex and closed  subset  of  $E$ and $T:S \to S$ is a continuous function such that $T(S)$ is relatively compact. Then  $T$  has  a  fixed  point  in  $S.$
\end{theorem}

\section{Construction of an appropriate phase space}

We need a suitable vector space $\mathcal{B} \subset G((-\infty,0],\mathbb{R}^n)$ equipped with a norm $\left\|\cdot\right\|_{\mathcal{B}}$ which satisfies the following axioms:

\begin{enumerate}

	\item[\namedlabel{a1}{(A1)}] $\mathcal{B}$ is complete.
	
	\item[\namedlabel{A2}{(A2)}] If $t_0 \in \mathbb{R}$, $\sigma >0$, $y\colon (-\infty,t_0+\sigma ]\to \mathbb{R}^n$ is regulated on $[t_0,t_0+\sigma]$ and $y_{t_0} \in \mathcal{B}$, then there are locally bounded functions $k_1,k_2,k_3\colon [0,\infty) \to (0,\infty)$, all independent of $y$, $t_0$ and $\sigma$, such that, 
	for every $t \in [t_0,t_0+\sigma]$:
	\begin{enumerate}[label=(\alph*)]
		
		\item $y_t \in \mathcal{B}$.
		
		\item $\left\|y(t)\right\| \leq k_1(t-t_0)\left\|y_t\right\|_{\mathcal{B}}$.
		
		\item $\left\|y_t\right\|_{\mathcal{B}} \leq k_2(t-t_0)\|y_{t_{0}}\|_{\mathcal{B}} + k_3(t-t_0)\displaystyle\sup_{u \in [t_0,t]}\left\|y(u)\right\|$.
		
	\end{enumerate}	
	
	\item[\namedlabel{A3}{(A3)}] For $t \geq 0$, let $S(t)\colon  \mathcal{B} \to \mathcal{B}$ be the operator defined by 
	$$(S(t)\varphi)(\theta) = \begin{cases} \varphi(0), & \theta = 0, \\
		\varphi(0^-), & -t \leq \theta <0, \\
		\varphi(t+\theta), & \theta <-t.  
	\end{cases} $$ Then, there is a continuous function $k\colon [0,\infty) \to [0,\infty)$ such that $k(0) = 0$ and $$ \left\|S(t)\varphi\right\|_{\mathcal{B}} \leq (1+k(t))\left\|\varphi \right\|_{\mathcal{B}}, \quad \text{for all } \varphi \in \mathcal{B}.$$
	
\end{enumerate} 
 
 The example below illustrates the existence of at least one set with all properties listed above.
 
 \begin{example}[\protect{\cite[Example 3.2]{15}}]
 	\label{20}
 	Let $\rho\colon  (-\infty,0] \to (0,\infty)$ be a continuous function such that $\rho(0) = 1$ and that the function $p \colon  [0, \infty) \to (0, \infty)$ given by $$p(t) = \sup_{\theta \leq -t} \frac{\rho(t+\theta)}{\rho(\theta)}, \quad t \geq 0,$$
 is locally bounded. The space $$\mathcal{B} = BG_{\rho} ((-\infty,0],\mathbb{R}^n) = \left\{\varphi \in G((-\infty,0],\mathbb{R}^n) \colon  \frac{\|\varphi(\theta)\|}{\rho(\theta)} \text{ is bounded} \right\}$$ endowed with the norm $$\|\varphi\|_{\rho} = \sup_{\theta \leq 0}\frac{\|\varphi (\theta)\|}{\rho (\theta)}, \quad \varphi \in BG_{\rho} ((-\infty,0],\mathbb{R}^n),$$ satisfies all the properties \ref{a1}--\ref{A3}. Therefore, $BG_{\rho} ((-\infty,0],\mathbb{R}^n)$ is a phase space.
\end{example}
 
In order to use some results of the Perron--Stieltjes integral, we need two lemmas:

\begin{lemma}[\protect{\cite[Lemma 3.8]{15}}]
	\label{e10}
	Assume that $\mathcal{B}$ is a phase space. If $y\colon  (-\infty, t_0 + \sigma] \to \mathbb{R}^n$ is such that $y_{t_0} \in \mathcal{B}$ and $y|_{[t_0,t_0+\sigma]}$ is a regulated function, then $t\mapsto\|y_t\|_{\mathcal{B}} $ is regulated on $[t_0,t_0 + \sigma]$.
\end{lemma}

\begin{lemma}[\protect{\cite[Lemma 3.10]{15}}]\label{80}
	Let $r\colon [t_0,t_0+\sigma] \to \mathbb{R}$ be a nondecreasing function such that $r(s) \leq s$ for all $s\in [t_0,t_0+\sigma]$. Assume that $y\colon  (-\infty, r(t_0 + \sigma)] \to \mathbb{R}^n$ is such that $y_{r(t_0)} \in \mathcal{B}$ and $y|_{[r(t_0),r(t_0+\sigma)]}$ is a regulated function, then $t\mapsto\|y_{r(t)}\|_{\mathcal{B}} $ is regulated on $[t_0,t_0 + \sigma]$.
\end{lemma}


\section{Existence and uniqueness of solutions}

In this section, our goal is to prove results concerning existence and uniqueness of solutions of measure functional differential equations with state--dependent delays given by
\begin{equation}
	\begin{array}{rcl}
		\label{eq3.1}
		x(t) &=& x(t_0)+ \displaystyle\int_{t_0}^{t} f(s,x_{\rho(s,x_s)}){\rm d}g(s), \quad t \in [t_0,t_0+\sigma], \\ x_{t_0}&=&\phi,
	\end{array}
\end{equation}
where $\sigma>0$, $t_0 \in \mathbb{R}$, $g\colon [t_0,t_0+\sigma] \to \mathbb{R}$ is a  nondecreasing function, $\mathcal{B} \subset G((-\infty,0],\mathbb{R}^n)$ is a Banach space satisfying axioms \ref{a1}--\ref{A3}, $\phi \in \mathcal{B}$ and $x\colon (-\infty,t_0+\sigma] \to \mathbb{R}^n$, $\rho\colon [t_0,t_0+\sigma]\times \mathcal{B} \to \mathbb{R}$ and 
$f\colon  [t_0,t_0+\sigma]\times \mathcal{B} \to \mathbb{R}^n$ are functions. 

To show that the problem $(\ref{eq3.1})$ has a solution, we begin by considering the set 
\begin{equation}
	\label{e11}
	X = \left\{x\colon (-\infty, t_0+\sigma] \to \mathbb{R}^n\colon  x_{t_0} \in \mathcal{B} \text{ and } x|_{[t_0,t_0+\sigma]} \text{ is regulated}\right\},
\end{equation} 
equipped with the norm
\begin{equation}
	\label{e12}
	\|x\|_{X} = \|x_{t_0}\|_{\mathcal{B}}+ \sup_{u\in [t_0,t_0+\sigma]} \|x(u)\|. 
\end{equation} 
This set is a Banach space. We also assume the following assumptions:
\begin{enumerate}
	
	\item[\namedlabel{(H1)}{(B1)}] For all $x \in \mathcal{B}$, the integral $\int_{t_0}^{t_0+\sigma} f(s,x) {\rm d}g(s)$ exists in the sense of Perron--Stieltjes.
	
	\item[\namedlabel{(H2)}{(B2)}]  There exists a Perron--Stieltjes integrable function $M\colon [t_0,t_0+\sigma] \to \mathbb{R}^+$ such that $$\left\|\int_{u_1}^{u_2} f(s,x) {\rm d}g(s)\right\| \leq \int_{u_1}^{u_2} M(s) {\rm d}g(s)$$ whenever $x \in \mathcal{B}$ and $u_1, u_2 \in [t_0,t_0+\sigma]$.
	
	\item[\namedlabel{(H3)}{(B3)}] There exists a regulated function $L\colon [t_0,t_0+\sigma] \to \mathbb{R}^+$ such that $$\left\|\int_{u_1}^{u_2} \left(f(s,x)-f(s,y)\right) {\rm d}g(s)\right\| \leq \int_{u_1}^{u_2} L(s) \left\|x-y\right\|_{\mathcal{B}}{\rm d}g(s)$$ whenever $x,y \in \mathcal{B}$  
	and $u_1, u_2 \in [t_0,t_0+\sigma]$ . 
	
	\item[\namedlabel{(H4)}{(B4)}]  There exists a regulated function $L_2\colon [t_0,t_0+\sigma] \to \mathbb{R}^+$ such that $$\left\|\int_{u_1}^{u_2} \left(f(s,x_{u})-f(s,x_{v})\right) {\rm d}g(s)\right\| \leq \int_{u_1}^{u_2} L_2(s) \left|u-v\right| {\rm d}g(s)$$ for all $x \in X$  
	and $u_1, u_2,u, v \in [t_0,t_0+\sigma]$.
	
	\item[\namedlabel{(H5)}{(B5)}] For all $x \in X$, 
	the function $t \mapsto \rho(t, x_t)$, $t \in [t_0,t_0+\sigma]$, is nondecreasing, satisfies $\rho(t,x_t) \leq t$ and $x_{\rho(t_0,x_{t_0})} \in \mathcal{B}$.
	
	\item[\namedlabel{(H6)}{(B6)}] There exists a regulated function $L_3\colon [t_0,t_0+\sigma] \to \mathbb{R}^+$ such that $$\int_{u_1}^{u_2} \left|\rho(s,x) - \rho(s,y)\right| {\rm d}g(s) \leq \int_{u_1}^{u_2} L_3(s) \left\|x-y\right\|_{\mathcal{B}} {\rm d}g(s)$$ 
	for all $u_1,u_2 \in [t_0,t_0+\sigma]$ and all $x,y \in \mathcal{B}$.

\end{enumerate}
\begin{remark}
	\label{81} By the properties of the integral and by Lemma \ref{e10} guarantee that, whenever $x\colon (-\infty,t_0+\sigma] \to \mathbb{R}^n$ is such that $x|_{[t_0,t_0+\sigma]}$ is regulated and $x_{t_0} \in \mathcal{B}$, the function $t \mapsto \| x_{t}\|_{\mathcal{B}}$ is Perron--Stieltjes integrable with respect to a nondecreasing function $g$.
\end{remark}
\begin{remark}
	Notice that condition \ref{(H5)} is necessary in order to ensure that $\|x_{\rho(t,x_t)}\|_{\mathcal{B}}$ is a regulated function (Lemma \ref{80}). Thus, in this case, following the same arguments used in the Remark \ref{81}, $\|x_{\rho(t,x_t)}\|_{\mathcal{B}}$ is Perron--Stieltjes integrable with respect to a nondecreasing function $g$, whenever $x\colon (-\infty,t_0+\sigma] \to \mathbb{R}^n$ is such that $x|_{[t_0,t_0+\sigma]}$ is regulated and $x_{t_0} \in \mathcal{B}$.
\end{remark}


It follows a result concerning the existence of solutions of measure FDEs with state--dependent delays.

\begin{theorem}[Existence of Solutions]
	\label{48}	
	Let $\mathcal{B} \subset G((-\infty,0],\mathbb{R}^n)$ be a Banach space satisfying axioms \ref{a1}--\ref{A3}, $\phi \in \mathcal{B}$ and $g\colon [t_0,t_0+\sigma] \to \mathbb{R}$ be a nondecreasing function. If $f\colon [t_0,t_0+\sigma] \times \mathcal{B} \to \mathbb{R}^n$ and $\rho\colon [t_0,t_0+\sigma] \times \mathcal{B}\to \mathbb{R}$ are functions that satisfy the properties \ref{(H1)}--\ref{(H6)}, then the problem \eqref{eq3.1}
	has a solution.
\end{theorem}
\begin{proof}
	Let
	\begin{equation}
		\label{A}
		A = \left\{x \in X\colon x_{t_0} = \phi \text{ and } \left\|x(t)-\phi(0)\right\|\leq \int_{t_0}^{t} M(s) {\rm d}g(s) \text{ for all } t \in [t_0,t_0+\sigma]\right\}
	\end{equation} and define the operator $\Gamma\colon  A\to X $ by $$\Gamma x(t) \coloneqq \begin{cases} \phi(t-t_0), & \text{if } t \leq t_0 , \\ x(t_0) + \displaystyle\int_{t_0}^{t} f(s,x_{\rho(s,x_s)}){\rm d}g(s), & \text{if } t_0 \leq t \leq t_0+\sigma. \end{cases} $$
	
	
	\noindent \textbf{Statement 1:}\label{A1} The set $A$ is convex.
	In fact, given $x,y \in A$, $\theta \in (-\infty,0]$ and $\xi \in (0,1)$, we have \begin{align*}
		\left(\xi x+(1-\xi)y\right)_{t_0}(\theta) & = \xi x(t_0 + \theta) + (1-\xi)y(t_0 + \theta) = \phi(\theta). 
	\end{align*}
	For all $t \in [t_0,t_0+\sigma]$, we get
	\begin{align*}
		\left\|\left(\xi x+(1-\xi)y\right)(t) -\phi(0)\right\| & \leq \xi\left\|x(t)-\phi(0)\right\|+(1-\xi)\left\|y(t)-\phi(0)\right\| \leq \int_{t_0}^{t} M(s) {\rm d}g(s),
	\end{align*}
	proving the Statement 1.
	
	\noindent \textbf{Statement 2:} $\Gamma(A) \subset A$.
	Indeed, for $x \in A$, we have $\left(\Gamma x\right)_{t_0}(\theta) = \left(\Gamma x\right)(t_0+\theta)  = \phi(\theta).$ 
	By \ref{(H2)}, we get
	\begin{align*}
		\left\|\Gamma x(t) - \phi(0)\right\| 
		= \left\|\int_{t_0}^{t}f(s,x_{\rho(s,x_s)}){\rm d}g(s)\right\| 
		\leq \int_{t_0}^{t} M(s) {\rm d}g(s),
	\end{align*}
	for all $t \in [t_0,t_0+\sigma]$, proving the Statement 2.
	
	\noindent	\textbf{Statement 3:} The set $A$ is bounded and closed. Indeed, let $(x_n)_{n\in \mathbb{N}}$ be a sequence in $A$ such that converges to $x$ on $\left\|\cdot\right\|_{X}$ norm. Then, for all $n \in \mathbb{N}$, $(x_n)_{t_0} = \phi$, $$\left\|x_n(t) - \phi(0)\right\| \leq \int_{t_0}^{t} M(s) {\rm d}g(s) \quad \text{ for all } t \in [t_0,t_0+\sigma]$$ 
	and	\begin{align}
		\left\|x_{t_0}-\phi\right\|_{\mathcal{B}} & \leq \left\|(x-x_n)_{t_0}\right\|_{\mathcal{B}} + \sup_{u \in [t_0,t_0+\sigma]} \left\|x(u)-x_n(u)\right\| \label{c} = \left\|x-x_n\right\|_{X}.
	\end{align}
	Thus, passing (\ref{c}) to limit when $n \to \infty$, we obtain $x_{t_0} = \phi$. By \ref{A2}, we have
	\begin{align}
		\label{C}
		\left\|x(t)-\phi(0)\right\| & \leq \left\|x(t)-x_n(t)\right\|+\left\|x_n(t)-\phi(0)\right\| \nonumber \\[4.5pt]
		& \leq k_1(t-t_0)\left\|(x-x_n)_t\right\|_{\mathcal{B}}+ \int_{t_0}^{t} M(s) {\rm d}g(s) \nonumber \\
		& \leq \sup_{u \in [0,\sigma]}k_1(u) \sup_{u \in [0,\sigma]}k_3(u) \left\|x-x_n\right\|_{X}+\int_{t_0}^{t} M(s) {\rm d}g(s)
	\end{align}
	for all $t \in [t_0,t_0+\sigma]$ and for all $n \in \mathbb{N}$. If $\sup_{u \in [0,\sigma]}k_1(u) \sup_{u \in [0,\sigma]}k_3(u) > 0$, then let $\varepsilon >0$ be arbitrary and $n_0 \in \mathbb{N}$ be such that $\left\|x-x_n\right\|_{X} < \varepsilon \left(\sup_{u \in [0,\sigma]}k_1(u) \sup_{u \in [0,\sigma]}k_3(u)\right)^{-1} $ for all $n\geq n_0$. By $(\ref{C})$, we have $$\left\|x(t)-\phi(0)\right\| < \varepsilon+\int_{t_0}^{t} M(s) {\rm d}g(s) \quad \text{for all } t\in [t_0,t_0+\sigma].$$	Since $\varepsilon$ is arbitrary, we conclude that \begin{equation}\label{D}\left\|x(t)-\phi(0)\right\|\leq \int_{t_0}^{t} M(s) {\rm d}g(s) \quad \text{for all } t\in[t_0,t_0+\sigma].\end{equation}
	Clearly, \eqref{D} is true if $\sup_{u \in [0,\sigma]}k_1(u) \sup_{u \in [0,\sigma]}k_3(u) = 0$. Thus, we obtain that $A$ is closed. Finally, by $(\ref{D})$, 
	\begin{align*}
		\left\|x\right\|_{X} & \leq \left\| \phi \right\|_{\mathcal{B}} + \sup_{u \in [t_0,t_0+\sigma]} \left( \left\|x(u) - \phi(0)\right\| +\|\phi(0)\| \right) \\
		& \leq \left\|\phi\right\|_{\mathcal{B}} +\int_{t_0}^{t_0+\sigma} M(s) {\rm d}g(s) + \|\phi(0)\|. 
	\end{align*}
	Therefore, $A$ is bounded and the statement is proved.	
	
	
	\noindent \textbf{Statement 4:} The operator $\Gamma$ is continuous. Firstly, given $x,y \in A$ and $t \in [t_0,t_0+\sigma]$, inequalities \ref{(H3)} and \ref{(H4)} imply that
	\begin{align}
		&\left\|\left(\Gamma x-\Gamma y\right)(t)\right\| = \left\|\int_{t_0}^{t} f(s,x_{\rho(s,x_s)})-f(s,y_{\rho(s,y_s)}) {\rm d}g(s)\right\| \nonumber\\
		& \leq \left\|\int_{t_0}^{t} f(s,x_{\rho(s,x_s)})-f(s,y_{\rho(s,x_s)}) {\rm d}g(s)\right\| +\left\|\int_{t_0}^{t} f(s,y_{\rho(s,x_s)})-f(s,y_{\rho(s,y_s)}) {\rm d}g(s)\right\| \nonumber\\
		& \leq \int_{t_0}^{t} L(s) \left\| x_{\rho(s,x_s)}-y_{\rho(s,x_s)}\right\|_{\mathcal{B}}{\rm d}g(s) +\int_{t_0}^{t} L_2(s)\left|\rho(s,x_s)-\rho(s,y_s)\right| {\rm d}g(s). \label{98}
	\end{align} 
	By axiom \ref{A2} and by inequalities \ref{(H6)} and \eqref{98}, we have
	\begin{align}
		& \left\|\left(\Gamma x-\Gamma y\right)(t)\right\| \nonumber\\
		&\leq  \int_{t_0}^{t} L(s) k_3(\rho(s,x_s)-t_0)\sup_{u\in [t_0,\rho(s,x_s)]}\hspace{-0.5cm}\left\| (x-y)(u)\right\| {\rm d}g(s) +\int_{t_0}^{t}L_2(s)L_3(s)\left\|x_s-y_s\right\|_{\mathcal{B}} {\rm d}g(s) \nonumber\\
		& \leq  \int_{t_0}^{t} \hspace{-0.2cm} L(s) k_3(\rho(s,x_s)-t_0)\left\| x-y\right\|_X{\rm d}g(s) \hspace{-0.05cm}+\hspace{-0.1cm}\int_{t_0}^{t}\hspace{-0.2cm} L_2(s)L_3(s)k_3(s-t_0) \hspace{-0.1cm} \sup_{u \in [t_0,s]} \hspace{-0.1cm} \left\| (x-y)(u)\right\| {\rm d}g(s) \nonumber \\
		& \label{f}\leq \int_{t_0}^{t}  \left( L(s) +L_2(s)L_3(s)\right) \sup_{u \in [0,\sigma]}k_3(u) {\rm d}g(s) \left\| x-y\right\|_X. 
	\end{align}
	Therefore, by (\ref{f}),
	\begin{align*}
		\left\|\Gamma x-\Gamma y\right\|_{X} & = \left\|\left(\Gamma x-\Gamma y\right)_{t_0}\right\|_{\mathcal{B}} + \sup_{u \in [t_0,t_0+\sigma]}  \left\|(\Gamma x-\Gamma y)(u)\right\| \\
		& \leq \int_{t_0}^{t_0+\sigma}  \left( L(s) +L_2(s)L_3(s)\right) \sup_{u \in [0,\sigma]}k_3(u) {\rm d}g(s) \left\| x-y\right\|_X,\end{align*}
	proving the continuity on $\left\|\cdot\right\|_{X}$ norm. 
	
	\noindent \textbf{Statement 5:}	The set $B\coloneqq\left\{f\colon [t_0,t_0+\sigma] \to \mathbb{R}^n\colon  f = \Gamma x |_{[t_0,t_0+\sigma]} \text{ for some } x\in A\right\}$ is relatively compact on $G([t_0,t_0+\sigma],\mathbb{R}^n)$. 
	Indeed, for all $t\in [t_0,t_0+\sigma]$,
	\begin{align*}
		\left\|\Gamma x(t)\right\| &\leq \left\|x(t_0)\right\|+\int_{t_0}^{t_0+\sigma} M(s) {\rm d}g(s).
	\end{align*}
	Furthermore,
	\begin{align*}
		\left\|\Gamma x(u) - \Gamma x(v)\right\| & \leq \int_{v}^{u} M(s) {\rm d}g(s).
	\end{align*}
	Since $g$ is nondecreasing, the function $h(t) = \int_{t_0}^{t} M(s) {\rm d}g(s)$ is nondecreasing. In addition, both functions $K\colon [t_0,t_0+\sigma] \to \mathbb{R}$ and $\eta\colon [0,\infty) \to [0,\infty)$ defined by $$K(t) =  h(t) + t, \quad \eta(t)=t$$ are increasing functions. Moreover, $\eta$ is continuous and $\eta(0) = 0$. By Theorem \ref{relativamente-compacto}, $B$ is relatively compact on $G([t_0,t_0+\sigma],\mathbb{R}^n)$.

	\noindent \textbf{Statement 6:} We conclude that $\Gamma$ is completely continuous.	In fact, let $(x_n)_{n\in \mathbb{N}} \subset A$ be a bounded sequence on $\left\|\cdot\right\|_{X}$ norm 
	and let $t \in [t_0,t_0+\sigma]$. By axiom \ref{A2}, we obtain
	\begin{align*}\left\|x_n(t)\right\| & \leq k_1(t-t_0)\left\|(x_n)_t\right\|_{\mathcal{B}} \\[6pt]
		& \leq k_1(t-t_0)\left( k_2(t-t_0)\left\|(x_n)_{t_0}\right\|_{\mathcal{B}}+k_3(t-t_0)\sup_{u\in [t_0,t]}\left\|x_n(u)\right\|\right) \\
		& \leq \sup_{u\in [0,\sigma]}k_1(u)\left( \sup_{u\in [0,\sigma]} k_2(u)\left\|(x_n)_{t_0}\right\|_{\mathcal{B}}+ \sup_{u\in [0,\sigma]} k_3(u)\sup_{u\in [t_0,t_0+\sigma]}\left\|x_n(u)\right\|\right) \\[7pt]
		& \leq D \|x_n\|_X,
	\end{align*}
	where $$D = \max\left\{\sup_{u\in [0,\sigma]}k_1(u)\sup_{u\in [0,\sigma]}k_2(u),\sup_{u\in [0,\sigma]}k_1(u)\sup_{u\in [0,\sigma]}k_3(u)\right\}.$$ This inequality proves that $(x_n)$ restricted to the interval $[t_0,t_0+\sigma]$ is bounded on the space $G([t_0,t_0+\sigma],\mathbb{R}^n)$. Consequently, by the last statement, there exists a subsequence $(x_{n_k})_{k\in \mathbb{N}}$ such that $(\Gamma(x_{n_k}))_{k\in \mathbb{N}}$ is convergent on $\left\|\cdot\right\|_{\infty}$ norm. If we denote its limit by $y $, then the function $\bar{y}\colon (-\infty,t_0+\sigma] \to \mathbb{R}^n$ given by	$$\bar{y}(t) = \begin{cases}
		\phi(t-t_0), & t \in (-\infty,t_0], \\
		y(t), & t \in [t_0,t_0+\sigma], 
	\end{cases} $$
	is well--defined and is such that 
	\begin{align}\left\|\Gamma(x_{n_k})-\bar{y}\right\|_{X}& \leq \left\|(\Gamma x_{n_k}-\bar{y})_{t_0}\right\|_{\mathcal{B}} + \sup_{u\in[t_0,t_0+\sigma]} \left\|(\Gamma(x_{n_k})-\bar{y})(u)\right\| \nonumber\\ & \leq \left\|(x_{n_k})_{t_0}-\phi\right\|_{\mathcal{B}} + \left\|\Gamma(x_{n_k})-y\right\|_{\infty} \label{d} = \left\|\Gamma(x_{n_k})-y\right\|_{\infty} .
	\end{align}
	Passing (\ref{d}) to limit when $k \to \infty$, we conclude that $(\Gamma(x_{n_k}))_{k}$ converges to $\bar{y}$ on $\left\|\cdot\right\|_{X}$ norm. Since $A$ is closed, $\bar{y} \in A$. We conclude that $\Gamma$ is completely continuous.	
	
	Finally, after all statements together with Theorem \ref{punto-fijo}, we conclude the desired result.
\end{proof}

In what follows, we present a result which ensures the uniqueness of solutions of \eqref{eq3.1}.
\begin{theorem}
	Let $\mathcal{B} \subset G((-\infty,0],\mathbb{R}^n)$ be a Banach space satisfying axioms \ref{a1}--\ref{A3}, $\phi \in \mathcal{B}$ and $g\colon [t_0,t_0+\sigma] \to \mathbb{R}$ be a nondecreasing and left--continuous function. If $f\colon [t_0,t_0+\sigma] \times \mathcal{B} \to \mathbb{R}^n$ and $\rho\colon [t_0,t_0+\sigma] \times \mathcal{B}\to \mathbb{R}$ are functions that satisfy the properties \ref{(H1)}--\ref{(H6)}, then the problem \eqref{eq3.1} possesses a unique solution on $(-\infty, t_0 + \sigma]$.
\end{theorem}
\begin{proof}
	If $x,y $ are solutions of (\ref{eq3.1}), then by following the same steps as in the proof of Theorem \ref{48}, we get
	\begin{eqnarray*}
		\left\|x(t)-y(t)\right\| & =  & \left\|\left(\Gamma x - \Gamma y\right)(t)\right\| \nonumber \\
		& \leq &  \int_{t_0}^{t} L(s) \sup_{u \in [0,\sigma]}k_3(u)\sup_{u\in [t_0,s]}\left\| (x-y)(u)\right\| {\rm d}g(s) + \\
		& \ \ + & \ \ \int_{t_0}^{t}L_2(s)L_3 (s) k_3(s-t_0) \sup_{u \in [t_0,s]} \left\| (x-y)(u)\right\| {\rm d}g(s)\nonumber \\
		&\label{g} \leq & \int_{t_0}^{t} (L(s)+L_2(s)L_3(s)) \sup_{u \in [0,\sigma]}k_3(u)\sup_{u\in [t_0,s]}\left\| (x-y)(u)\right\| {\rm d}g(s).
	\end{eqnarray*}
	Let $\psi(v) = \displaystyle\sup_{u \in [t_0,v]}\left\|x(u)-y(u)\right\|$. Since $x,y$ are regulated functions, it follows that $\psi$ is also regulated, and thus, Perron--Stieltjes integrable with respect to the function $g$. Therefore, 
	\begin{align*}
		\psi(t) & \leq \int_{t_0}^{t}  (L(s)+L_2(s)L_3(s)) \sup_{u \in [0,\sigma]}k_3(u)\psi(s) {\rm d}g(s) \leq  K\int_{t_0}^{t}\psi(s) {\rm d}g(s),
	\end{align*}
	where $$K =\sup_{u \in [t_0,t_0+\sigma]}(L(u)+L_2(u)L_3(u)) \sup_{u \in [0,\sigma]}k_3(u).$$ Applying Theorem \ref{Gronwall}, we get $\psi(t) \leq 0$. 
	Since $\psi(t) \geq 0$ by definition, it follows the desired result.
\end{proof}

In sequel, we exemplify three functions $f,g$ and $\rho$ and a phase space that fit all the hypotheses of Theorem \ref{48}.
\begin{example}
Let $q\colon (-\infty,0]\to \mathbb{R}$ be the function $q(\theta) = e^{\theta}$ and choose the phase space $ \mathcal{B} = BG_{q } ( ( -\infty,0],\mathbb{R})$ defined as in Example \ref{20}. Let $T\colon(-\infty,0]\to \mathbb{R}$ be a bounded continuous function such that:
\begin{enumerate}[label=(\alph*)]
	\item $\displaystyle\frac{T(\theta)}{q(\theta)}$ is bounded.
	\item $\int_{-\infty}^{0}|T(\theta)|q(\theta) {\rm d}\theta < \infty$.
	\item There exists a constant $D>0$ such that $\int_{-\infty}^{0}|T(\theta-t_2)-T(\theta-t_1)| {\rm d}\theta \leq D|t_2-t_1|$ for all $0 \leq t_1,t_2$.
\end{enumerate}
Define the functions $f\colon [0,\infty) \times \mathcal{B} \to \mathbb{R}$ and $\rho\colon [0,\infty)\times \mathcal{B} \to [0, \infty)$ by \begin{equation*}
	\label{22}
	f(t,x) = \cos^2(t)\int_{-\infty}^{0} T(\theta)\tanh(x(\theta)) {\rm d}\theta, \quad \quad \rho(t,x) = t-\int_{-\infty}^{-t} |T(\theta)|\tanh(|x(\theta-t)|){\rm d}\theta.
\end{equation*}
It is immediate that for all $x\in X$, the function $t \mapsto \rho(t,x_t)$, $t \in [0,\sigma]$, is nondecreasing, satisfies $\rho(t,x_t) \leq t$ and $x_{\rho(0,x_0)}\in \mathcal{B}$. In addition, by definition of $\rho$, 
\begin{align}
	\rho(s,y )-\rho(s,x )  & = \int_{-\infty}^{-s} |T(\theta)|\left(\tanh(|y(\theta-s)|) - \tanh(|x(\theta-s)|) \right) {\rm d}\theta \nonumber \\
	& = \int_{-\infty}^{-2s} |T(u+s)|\left(\tanh(|y(u)|) - \tanh(|x(u)|) \right) {\rm d}u 	 \label{74} 
\end{align} 
By \eqref{74}, we have 
\begin{align*}
	|\rho(s,y )-\rho(s,x )|  & \leq \int_{-\infty}^{0} |T(u+s)||\tanh(|y(u)|) - \tanh(|x(u)|) | {\rm d}u \nonumber \\
	&\leq \int_{-\infty}^{0} |T(u+s)|\frac{|y(u)-x(u)|}{q(u)} q(u)  {\rm d}u \nonumber \\
	& \leq  \sup_{\theta \leq 0} |T(\theta)| \int_{-\infty}^{0} q(u) {\rm d}u \|y-x\|_{\mathcal{B}}\,. 
\end{align*}

Now, since $|\tanh z| \leq 1$ for all $z \in \mathbb{R}$ and since there is constant $C>0$ such that $|T(\theta)|/q(\theta) \leq C$ for all $\theta \leq 0$, we get $$|f(t,x)| \leq \int_{-\infty}^{0} \frac{|T(\theta)|}{q(\theta)}|\tanh(x(\theta))|q(\theta) {\rm d}\theta \leq \int_{-\infty}^{0} Cq(\theta) {\rm d}\theta = C$$
for all $(t,x) \in [0,\infty)\times \mathcal{B}$. Additionally, if $(t,x),(s,y) \in [0,\infty)\times \mathcal{B}$, then
\begin{align}\label{26}
	f(t,x)-f(s,y) 	& = \cos^2(t)\int_{-\infty}^{0} T(\theta)\tanh(x(\theta)) {\rm d}\theta-\cos^2(s)\int_{-\infty}^{0} T(\theta)\tanh(y(\theta)) {\rm d}\theta \nonumber\\
	& = \cos^2(t)\int_{-\infty}^{0} \hspace{-0.3cm}T(\theta)(\tanh(x(\theta))-\tanh(y(\theta))) {\rm d}\theta \nonumber \\
	&\hspace{0.7cm} + (\cos^2(t)-\cos^2(s))\int_{-\infty}^{0}\hspace{-0.3cm} T(\theta)\tanh(y(\theta)) {\rm d}\theta\,.
\end{align}
In particular, for all $x \in X$ and all $ 0 \leq b \leq a \leq \sigma $,
\begin{align}\label{25}
	f(t,x_a)- f(s,x_b)  & = \cos^2(t)\int_{-\infty}^{0} \hspace{-0.3cm}T(\theta)(\tanh(x(\theta+a))-\tanh(x(\theta+b))) {\rm d}\theta \nonumber\\
	& \hspace{0.5cm}+ (\cos^2(t)-\cos^2(s))\int_{-\infty}^{0}\hspace{-0.3cm} T(\theta)\tanh(x(\theta+b)) {\rm d}\theta \nonumber\\
	&  = \cos^2(t) \left( \int_{-\infty}^{0} (T(u-a)-T(u-b))\tanh(x(u)){\rm d}u\right. \nonumber\\
	&\hspace{-1.0cm} \left. + \int_{-b}^{0} (T(u+b-a)-T(u))\tanh(x(u+b)){\rm d}u + \int_{b-a}^{0} T(u)\tanh(x(u+a)){\rm d}u \right)\nonumber\\
	&\hspace{0.5cm}+(\cos^2(t)-\cos^2(s))\int_{-\infty}^{0}\hspace{-0.3cm} T(\theta)\tanh(x(\theta+b)) {\rm d}\theta.
\end{align} By \eqref{26},
\begin{align*}
	|f(s,x)-f(s,y)| & \leq \int_{-\infty}^{0} |T(\theta)||\tanh(x(\theta))-\tanh(y(\theta))| {\rm d}\theta \nonumber\\
	&\leq \int_{-\infty}^{0} |T(\theta)|\frac{|x(\theta)-y(\theta)|}{q(\theta)}q(\theta) {\rm d}\theta  \nonumber\\
	& \leq  \int_{-\infty}^{0} |T(\theta)| q(\theta) {\rm d}\theta \;  \|x-y\|_{\mathcal{B}}.
\end{align*}
Then, $$\left|\int_{u_1}^{u_2} f(s,x)-f(s,y) {\rm d}s \right|\leq  \int_{u_1}^{u_2} \overline{C}  \|x-y\|_{\mathcal{B}} {\rm d}s,$$ for all $u_1,u_2 \geq 0$, where $\overline{C} = \int_{-\infty}^{0} |T(\theta)| q(\theta) {\rm d}\theta$. By \eqref{25}, we obtain
\begin{align*}
	|f(s,x_a)-f(s,x_b) |& \leq \int_{-\infty}^{0} |T(u-a)-T(u-b)|{\rm d}u \nonumber \\
	&\hspace{0.6cm} + \int_{-\infty}^{0} |T(u+b-a)-T(u)|{\rm d}u + \int_{b-a}^{0} |T(u)||\tanh(x(u+a))|{\rm d}u  \nonumber\\
	& \leq 2D|a-b|+\overline{D}|a-b|,
\end{align*} where $\overline{D} =\sup_{\theta \leq 0} |T(\theta)|$. Therefore, all hypotheses of Theorem \ref{48} are satisfied for the case where the function $g\colon[0,\sigma] \to \mathbb{R}$ is given by $g(s) = s$. The continuity of $g$ is enough to conclude that	\begin{equation*}
	\begin{array}{rcl}
		x(t) &=& x(t_0)+ \displaystyle\int_{t_0}^{t} f(s,x_{\rho(s,x_s)}){\rm d}g(s), \quad t \in [t_0,t_0+\sigma], \\ x_{t_0}&=&\phi,
	\end{array}
\end{equation*}
has a unique solution.
\end{example}

\begin{remark}
	Since the example above uses an abstract function $T\colon (-\infty,0]\to \mathbb{R}$ that satisfies some assertions, the question of the existence of such a function arises. Indeed, it is possible to verify that the function $T(\theta)=e^{-\theta^2+\theta}$ answers positively this question.
\end{remark}


\section{Periodic averaging principle} \label{teoremaprincipalmesquita}

In this section, our goal is to prove a periodic averaging theorem for measure functional differential equation with state--dependent delays.

Let $\varepsilon_0>0$, $L>0$, $T > 0$. Consider a pair of bounded functions $f \colon [0,\infty) \times \mathcal{B} \to \mathbb{R}^n$, $g \colon [0,\infty) \times \mathcal{B} \times (0,\varepsilon_0] \to \mathbb{R}^n$, a left-continuous nondecreasing function $h \colon [0,\infty)\to \mathbb{R}$ and a function $\rho \colon [0,\infty)\times\mathcal{B}  \times (0,\varepsilon_0]\to [0,\infty)$ such that the following conditions are satisfied:

\begin{enumerate}
	
	\item[\namedlabel{M1}{(C1)}] For all $x \in \mathcal{B}$, the integrals $\int_{u_1}^{u_2} f(s,x) {\rm d}h(s)$ and $\int_{u_1}^{u_2} g(s,x, \varepsilon) {\rm d}h(s)$ exist
for all $u_1, u_2 \in [0, +\infty)$ and $\varepsilon \in (0,\varepsilon_{0}]$ in the sense of Perron--Stieltjes.
	\item[\namedlabel{M2}{(C2)}] $f$ is $T$--periodic with respect to the first variable.
	
	\item[\namedlabel{M3}{(C3)}] There exists a constant $\alpha >0 $ such that $h(t+T)-h(t) = \alpha$ for all $t \geq 0$.
	
	\item[\namedlabel{M4}{(C4)}] There exists a constant $C >0 $ such that 
$$\left\| \int_{u_1}^{u_2} [f(s, x)-f(s, y)] {\rm d} h(s)\right\| \leq \int_{u_1}^{u_2} C\left\|x -y\right\|_{\mathcal{B}} {\rm d}h(s)$$
for all $x,y\in \mathcal{B}$ and $u_1, u_2 \in [0, +\infty)$.
	
	\item[\namedlabel{M5}{(C5)}] The integral $$f_0(x)= \frac{1}{T} \int_{0}^{T} f(s, x) {\rm d}h (s)$$ exists in the sense of Perron--Stieltjes for all $x \in \mathcal{B}$.
	
\item[\namedlabel{M6}{(C6)}] For all $x \colon (-\infty,\alpha] \to \mathbb{R}^n$, $\alpha >0$, such that $x_0 = \phi \in \mathcal{B}$ and $x|_{[0,\alpha]}$ is regulated, there exists a constant $C_2>0$ such that 
$$\left\| \int_{u_1}^{u_2} [f(s, x_a)-f(s, x_b)]{\rm d} h(s)\right\| \leq \int_{u_1}^{u_2} C_2\left|a-b\right| {\rm d}h(s)$$ 
for all $x \in \mathcal{B}$ and $a, b, u_1, u_2 \in [0, \infty)$.

	\item[\namedlabel{M7}{(C7)}] For all $\varepsilon \in (0, \varepsilon_0]$, $t \geq 0$ and $x \in \mathcal{B}$, $\rho(t,x, \varepsilon) \leq t$ and $t \mapsto \rho(t, x_t, \varepsilon)$ is a regulated function and if $x$ is regulated, then $x_{\rho(t,x_t)}$ is also regulated with respect to $t$. 	
	
	\item[\namedlabel{M8}{(C8)}] For all $x:(-\infty,\alpha] \to \mathbb{R}^n$, $\alpha >0$, such that $x_0 = \phi \in \mathcal{B}$ and $x|_{[0,\alpha]}$ is regulated, there exists a constant $C_3>0$ such that 
$$\left| \rho(t, x_a, \varepsilon) - \rho(t, x_b, \varepsilon) \right| \leq \varepsilon C_3 |a - b| $$ for all $a, b, t \in [0,\infty)$, $x \in \mathcal{B}$ and $\varepsilon \in (0, \varepsilon_0]$. 

\item[\namedlabel{M9}{(C9)}] There exists a constant $C_4>0$ such that $$\left|\rho(s,y, \varepsilon)-\rho(s,x, \varepsilon)\right| \leq C_4\left\|y- x\right\|_{\mathcal{B}}$$ for all $s \in [0,\infty)$, $\varepsilon \in (0, \varepsilon_0]$ and $x, y \in \mathcal{B}$.

\end{enumerate}

Now, we are ready to prove our periodic averaging theorem for measure FDEs with state--dependent delays. 

\begin{theorem}\label{3}
	Let $\varepsilon_0 > 0$, $\mathcal{B} \subset G((-\infty,0],\mathbb{R}^n)$ be a Banach space satisfying axioms \ref{a1}--\ref{A3}, $h:[0, +\infty) \to \mathbb{R}$ be a left--continuous nondecreasing function and $\phi \in \mathcal{B}$. Assume that $f:[0, +\infty) \times \mathcal{B} \to \mathbb{R}^n$, $g: [0, +\infty) \times \mathcal{B} \times (0, \varepsilon_0] \to \mathbb{R}^n$ are bounded functions and $\rho: [0, +\infty) \times \mathcal{B}\to [0, +\infty)$ is a function. Also, suppose that the properties \ref{M1}--\ref{M9} are satisfied. Suppose that for all $\varepsilon \in (0,\varepsilon_0]$, the initial value problems
	\begin{equation}\label{eq-1}
\left\{
	\begin{array}{rcl}
	x(t) &=& \displaystyle x(0) +\varepsilon\int_{0}^{t} f\left(s, x_{\rho(s,x_s,\varepsilon)}\right){\rm d}h(s) + \varepsilon^2\int_{0}^{t}g\left(s, x_{\rho(s,x_s,\varepsilon)},\varepsilon \right){\rm d}h(s), \\ x_{0}&=& \phi,
	\end{array}
\right.
	\end{equation}
and 
\begin{equation}\label{eq-2}
\left\{
	\begin{array}{rcl}
	y(t) &=& \displaystyle y(0) +\varepsilon\int_{0}^{t} f_0\left(y_{\rho(s,y_s, \varepsilon)}\right){\rm d}s \\ y_{0}&=& \phi,
	\end{array}
\right.
	\end{equation}
	have solutions $x^{\varepsilon},y^{\varepsilon} : (-\infty,L/\varepsilon] \to \mathbb{R}^n$. Then there exists a $J>0$ such that the inequality
\begin{equation}\label{eq-teorema}
\left\|x^{\varepsilon}(t) - y^{\varepsilon}(t)\right\|_X \leq J\varepsilon  \ \ \textrm{for all} \ \ t \in (-\infty, L/\varepsilon]
\end{equation}
holds. 
	
\end{theorem}

\begin{proof}
		Since $f:[0,\infty) \times \mathcal{B} \to \mathbb{R}^n$ and $g:[0,\infty)\times \mathcal{B} \times (0,\varepsilon_{0}] \to \mathbb{R}^n $ are bounded, there exists $M>0$ such that $\left\|f(t,x)\right\|\leq M $ and $\left\|g(t, x, \varepsilon)\right\|\leq M$ for all $x \in \mathcal{B}$, $t \geq 0$ and $\varepsilon \in (0,\varepsilon_{0}]$. Then, for all $x \in \mathcal{B}$, we have
	\begin{equation}\label{f0}
	\left\|f_0(x)\right\|  \leq \frac{1}{T} \int_{0}^{T} \left\|f(s, x)\right\|{\rm d}h(s) \leq \frac{M}{T}(h(T)-h(0)) =  \frac{M}{T}\alpha.
	\end{equation}
For $t \in [0,L/\varepsilon]$, if $x^{\varepsilon}$ and $y^{\varepsilon}$ are solutions of (\ref{eq-1}) and (\ref{eq-2}), then
\begin{equation}\label{eq-norma-X}
\|x^\varepsilon (t) - y^{\varepsilon} (t)\|_X   =   \| (x^\varepsilon - y^\varepsilon)_{0}\|_{\mathcal{B}} + \sup_{t \in [0, L/\varepsilon]} \|x^{\varepsilon} (t) - y^{\varepsilon} (t)\|=  \sup_{t \in  [0, L/\varepsilon]} \|x^{\varepsilon} (t) - y^{\varepsilon} (t)\|. 
\end{equation}
On the other hand, for $t \in [0, L/\varepsilon]$, by the conditions \ref{M1}--\ref{M9}, we get
$$
\left\| x^{\varepsilon}(t) - y^{\varepsilon}(t)\right\| $$
$$=   \left\|\varepsilon\int_{0}^{t} f\left(s, x_{\rho(s,x_s^{\varepsilon}, \varepsilon)}^{\varepsilon}\right){\rm d}h(s) + \varepsilon^2\int_{0}^{t}g\left(s, x_{\rho(s,x_s^{\varepsilon}, \varepsilon)}^{\varepsilon}, \varepsilon \right){\rm d}h(s)-\varepsilon\int_{0}^{t} f_0\left(y_{\rho(s,y_s^{\varepsilon},\varepsilon)}^{\varepsilon}\right){\rm d}s  \right\| $$
$$ \leq  \varepsilon\left\|\int_{0}^{t} f\left(s, x_{\rho(s,x_s^{\varepsilon}, \varepsilon)}^{\varepsilon}\right) - f\left(s, y_{\rho(s,x_s^{\varepsilon}, \varepsilon)}^{\varepsilon}\right){\rm d}h(s)\right\|+ \varepsilon\left\|\int_{0}^{t} f\left(s, y_{\rho(s,x_s^{\varepsilon}, \varepsilon)}^{\varepsilon}\right) - f\left(s, y_{\rho(s,y_s^{\varepsilon}, \varepsilon)}^{\varepsilon}\right){\rm d}h(s)\right\|$$
$$ + \ \varepsilon\left\|\int_{0}^{t}f\left(s, y_{\rho\left(s,y_s^{\varepsilon}, \varepsilon\right)}^{\varepsilon} \right){\rm d}h(s) -\int_{0}^{t} f_0\left(y_{\rho(s,y_s^{\varepsilon}, \varepsilon)}^{\varepsilon}\right){\rm d}s  
	\right\| +\varepsilon^2M(h(t)-h(0)) $$
 $$\leq \varepsilon \int_{0}^{t} C \left\| x_{\rho(s,x_s^{\varepsilon}, \varepsilon)}^{\varepsilon}- y_{\rho(s,x_s^{\varepsilon}, \varepsilon)}^{\varepsilon}\right\|_{\mathcal{B}} {\rm d}h(s)+ \varepsilon\int_{0}^{t}C_2 \left|\rho(s,x_s^{\varepsilon}, \varepsilon) -  \rho(s,y_s^{\varepsilon}, \varepsilon)\right|{\rm d}h(s)$$
$$+ \varepsilon\left\|\int_{0}^{t} f\left(s, y_{\rho(s,x_s^{\varepsilon}, \varepsilon)}^{\varepsilon}\right){\rm d}h(s) - \int_{0}^{t} f_0\left(y_{\rho(s,y_s^{\varepsilon}, \varepsilon)}^{\varepsilon}\right){\rm d}s\right\|+\varepsilon^2M(h(t)-h(0)) $$
$$\leq \varepsilon C\int_{0}^{t} k_3(s) \sup_{u\in [0,s]} \left\| x^{\varepsilon}(u)- y^{\varepsilon}(u)\right\| {\rm d}h(s)+ \varepsilon C_2  C_4\int_{0}^{t}\left\| x_s^{\varepsilon} -  y_s^{\varepsilon} \right\|_{\mathcal{B}}{\rm d}h(s)$$
$$ +\varepsilon\left\|\int_{0}^{t} f\left(s, y_{\rho(s,y_s^{\varepsilon}, \varepsilon)}^{\varepsilon}\right){\rm d}h(s) - \int_{0}^{t} f_0\left(y_{\rho(s,y_s^{\varepsilon}, \varepsilon)}^{\varepsilon}\right){\rm d}s\right\|+\varepsilon^2M(h(t)-h(0))$$
	$$\leq \varepsilon C\int_{0}^{t} k_3(s) \sup_{u\in [0,s]} \left\| x^{\varepsilon}(u)- y^{\varepsilon}(u)\right\| {\rm d}h(s)+ \varepsilon C_2  C_4\int_{0}^{t} k_3 (s) \sup_{u \in [0, s]}\left\| x^{\varepsilon} (u) -  y^{\varepsilon} (u) \right\|{\rm d}h(s)$$
$$ +\varepsilon\left\|\int_{0}^{t} f\left(s, y_{\rho(s,y_s^{\varepsilon}, \varepsilon)}^{\varepsilon}\right){\rm d}h(s) - \int_{0}^{t} f_0\left(y_{\rho(s,y_s^{\varepsilon}, \varepsilon)}^{\varepsilon}\right){\rm d}s\right\|+\varepsilon^2M(h(t)-h(0))$$
$$\leq \varepsilon (C + C_2 C_4)\int_{0}^{t} k_3(s) \sup_{u\in [0,s]} \left\| x^{\varepsilon}(u)- y^{\varepsilon}(u)\right\| {\rm d}h(s)$$
\begin{equation}\label{eq-4.5}
+\varepsilon\left\|\int_{0}^{t} f\left(s, y_{\rho(s,y_s^{\varepsilon}, \varepsilon)}^{\varepsilon}\right){\rm d}h(s) - \int_{0}^{t} f_0\left(y_{\rho(s,y_s^{\varepsilon}, \varepsilon)}^{\varepsilon}\right){\rm d}s\right\|+\varepsilon^2M(h(t)-h(0)).
\end{equation}

	Let $p$ is the largest integer such that $pT \leq t$. Therefore, we have
	$$\left\|\int_{0}^{t} f\left(s, y_{\rho(s,y_s^{\varepsilon}, \varepsilon)}^{\varepsilon}\right){\rm d}h(s) - \int_{0}^{t} f_0\left(y_{\rho(s,y_s^{\varepsilon}, \varepsilon)}^{\varepsilon}\right){\rm d}s\right\| $$
$$ \leq \sum_{i=1}^{p} \left\|\int_{(i-1)T}^{iT} f\left(s, y_{\rho(s,y_s^{\varepsilon}, \varepsilon)}^{\varepsilon}\right)-f\left(s, y_{\rho\left(s,y_{(i-1)T}^{\varepsilon}, \varepsilon\right)}^{\varepsilon}\right) {\rm d}h(s) \right\| $$
$$ + \sum_{i=1}^{p} \left\| \int_{(i-1)T}^{iT} f\left(s, y_{\rho\left(s,y_{(i-1)T}^{\varepsilon}, \varepsilon\right)}^{\varepsilon}\right){\rm d}h(s) - \int_{(i-1)T}^{iT} f_0\left(y_{\rho(s,y_{(i-1)T}^{\varepsilon}, \varepsilon)}^{\varepsilon}\right) {\rm d}s \right\| $$
$$ + \sum_{i=1}^{p} \left\| \int_{(i-1)T}^{iT} [f_0\left(y_{\rho(s,y_{(i-1)T}^{\varepsilon}, \varepsilon)}^{\varepsilon}\right)- f_0\left(y_{\rho(s,y_s^{\varepsilon}, \varepsilon)}^{\varepsilon}\right)]{\rm d}s\right\| $$
\begin{equation}
\label{e16}
	+ \left\|\int_{pT}^{t} f\left(s, y_{\rho(s,y_s^{\varepsilon}, \varepsilon)}^{\varepsilon}\right){\rm d}h(s) - \int_{pT}^{t} f_0\left(y_{\rho(s,y_s^{\varepsilon}, \varepsilon)}^{\varepsilon}\right){\rm d}s\right\|.
	\end{equation}
For every $i \in \{1,2,\ldots,p\}$ and every $s \in [(i-1)T,iT]$, we obtain
\begin{align*}
\sum_{i=1}^{p} \left\|\int_{(i-1)T}^{iT} f\left(s, y_{\rho(s,y_s^{\varepsilon}, \varepsilon)}^{\varepsilon}\right) \right. & \left. -f(s, y_{\rho(s,y_{(i-1)T}^{\varepsilon}, \varepsilon)}^{\varepsilon}) {\rm d}h(s) \right\| \\
&  \leq \sum_{i=1}^{p} \int_{(i-1)T}^{iT} C_2 |{\rho(s,y_s^{\varepsilon}, \varepsilon)} - {\rho(s,y_{(i-1)T}^{\varepsilon}, \varepsilon)}| {\rm d}h(s) \\
&\leq \sum_{i=1}^{p} C_2 C_3 \varepsilon \int_{(i-1)T}^{iT} |s - iT +T| {\rm d}h(s) \\
&  \leq \sum_{i=1}^{p} C_2C_3 T \varepsilon (h(iT) - h((i-1)T)) \\
& 	= C_2 C_3 T \alpha p \varepsilon.
\end{align*}
Using this estimate and the fact that $pT \leq L/\varepsilon$, we get
\begin{equation}\label{eq-4.5-1}
\sum_{i=1}^{p} \left\|\int_{(i-1)T}^{iT} \left(f\left(s, y_{\rho(s,y_s^{\varepsilon}, \varepsilon)}^{\varepsilon}\right)-f\left(s, y_{\rho(s,y_{(i-1)T}^{\varepsilon}, \varepsilon)}^{\varepsilon}\right) \right) {\rm d}h(s) \right\| \leq C_2 C_3 \alpha L.
\end{equation}
	On the other hand, notice that for $s \in [(i-1)T,iT]$, we have
\begin{align*}
\left\| f_0 \left(y_{\rho(s,y_s^{\varepsilon}, \varepsilon)}^{\varepsilon}\right) - f_0\left(y^{\epsilon}_{\rho(s,y_{(i-1)T}^{\varepsilon}, \varepsilon)}\right) \right\| & = \left\|\dfrac{1}{T}\int_{0}^{T} \left( f\left(u, y_{\rho(s,y_s^{\varepsilon}, \varepsilon)}^{\varepsilon}\right)-f\left(u, y_{\rho(s,y_{(i-1)T}^{\varepsilon}, \varepsilon)}^{\varepsilon}\right) \right) {\rm d}h(u) \right\| \\
& \leq \dfrac{1}{T}\int_{0}^{T} C_2 |{\rho(s,y_s^{\varepsilon}, \varepsilon)} - {\rho(s,y_{(i-1)T}^{\varepsilon}, \varepsilon)}| {\rm d}h(u) \\
&\leq \dfrac{1}{T}\int_{0}^{T} C_2  C_3 \varepsilon |s - (i-1)T|{\rm d}h(u) \\
&\leq C_2 C_3 \varepsilon \alpha.
\end{align*}
	Therefore, it implies that
\begin{eqnarray}
\sum_{i=1}^{p} \left\| \int_{(i-1)T}^{iT} f_0\left(y_{\rho(s,y_s^{\varepsilon}, \varepsilon)}^{\varepsilon}\right) - f_0\left(y_{\rho(s, y_{(i-1)T}^{\varepsilon}, \varepsilon)}^{\varepsilon}\right) {\rm d}s \right\| &\leq & \sum_{i=1}^{p}  \int_{(i-1)T}^{iT}  C_2 C_3 \varepsilon \alpha {\rm d}s \nonumber\\
&\leq & C_2 C_3 \varepsilon \alpha p T \nonumber\\
&\leq & C_2 C_3 \alpha L. \label{eq-4.5-2}
\end{eqnarray}
The fact that $f$ is $T$-periodic in the first variable and the definition of $f_0$ imply
	\begin{align}
	&\sum_{i=1}^{p} \left\| \int_{(i-1)T}^{iT} f\left(y_{\rho(s, y_{(i-1)T}^{\varepsilon}, \varepsilon)}^{\varepsilon},s\right) {\rm d}h(s) - \int_{(i-1)T}^{iT}f_0\left(y_{\rho(s, y_{(i-1)T}^{\varepsilon}, \varepsilon)}^{\varepsilon}\right){\rm d}s\right\| \nonumber\\
	& = \sum_{i=1}^{p} \left\|  \int_{0}^{T} f\left(y_{\rho(s, y_{(i-1)T}^{\varepsilon}, \varepsilon)}^{\varepsilon},s\right) {\rm d}h(s) - f_0\left(y_{\rho(s, y_{(i-1)T}^{\varepsilon}, \varepsilon)}^{\varepsilon}\right) T \right\| = 0. \label{eq-4.5-3}
	\end{align}
Finally, we have
$$
	\left\|\int_{pT}^{t} f\left(s, y_{\rho(s,y_s^{\varepsilon}, \varepsilon)}^{\varepsilon}\right){\rm d}h(s) - \int_{pT}^{t} f_0\left(y_{\rho(s,y_s^{\varepsilon}, \varepsilon)}^{\varepsilon}\right){\rm d}s\right\| $$
$$\leq  \left\|\int_{pT}^{t} f\left(s, y_{\rho(s,y_s^{\varepsilon}, \varepsilon)}^{\varepsilon}\right){\rm d}h(s)\right\| + \left\|\int_{pT}^{t} f_0\left(y_{\rho(s,y_s^{\varepsilon}, \varepsilon)}^{\varepsilon}\right){\rm d}s\right\| $$
$$ \leq  M (h(t)-h(pT)) + \frac{M \alpha}{T}(t-pT) $$
\begin{equation}\label{eq-4.5-4} 
 \leq M(h((p+1)T) - h(pT)) + \frac{M \alpha}{T}T = M\alpha + M\alpha = 2M\alpha.
\end{equation}
	Combining inequalities \eqref{e16}, \eqref{eq-4.5-1}, \eqref{eq-4.5-2}, \eqref{eq-4.5-3} and \eqref{eq-4.5-4}, we get 
\begin{eqnarray}
\hspace{-8mm} \left\|\int_{0}^{t} f\left(y_{\rho(s,y_s^{\varepsilon}, \varepsilon)}^{\varepsilon},s\right){\rm d}h(s) - \int_{0}^{t} f_0\left(y_{\rho(s,y_s^{\varepsilon}, \varepsilon)}^{\varepsilon}\right){\rm d}s\right\|  \hspace{-2mm} &\leq &  \hspace{-2mm} 2 M \alpha + C_2 C_3 \alpha L + C_2 C_3 \alpha L \nonumber\\
\hspace{-2mm} & \leq & \hspace{-2mm} 2 \alpha(M + C_2 C_3L). \label{eq-4.5-5}
\end{eqnarray}
	From inequalities \eqref{eq-4.5} and \eqref{eq-4.5-5}, we get
\begin{align*}
\left\|x^{\varepsilon}(t)\right.& \left.-y^{\varepsilon}(t) \right\|  \\
&\leq \varepsilon (C + C_2 C_4) \int_{0}^{t} k_3(s) \sup_{u \in [0,s]} \left\|y^{\varepsilon}(u) - x^{\varepsilon}(u)\right\| {\rm d}h(s)+\varepsilon K + \varepsilon^2 M (h(t)-h(0)),
\end{align*}
where $K= 2 \alpha(M + C_2 C_3L)$. Since $k_3$ is bounded, there exists $K'>0$ such that $\displaystyle\sup_{s\in [0,t] }k_3(s) \leq K'$. It implies that
	\begin{align*}
	\left\|x^{\varepsilon}(t) - y^{\varepsilon}(t) \right\| \leq \varepsilon (C + C_2 C_4) K'\hspace{-0.2cm} \int_{0}^{t} \hspace{-0.1cm}\sup_{u \in [0,s]} \left\|y^{\varepsilon}(u) - x^{\varepsilon}(u)\right\| {\rm d}h(s)+\varepsilon K + \varepsilon^2 M (h(t)-h(0)).
	\end{align*}
Notice that
\begin{eqnarray*}
\varepsilon (h(t)-h(0))  & \leq & \varepsilon \left(h\left(\frac{L}{\varepsilon}\right)-h(0)\right) \leq \varepsilon  \left(h\left(\left\lceil\frac{L}{\varepsilon T}\right\rceil T\right)-h(0)\right)  \leq \varepsilon \left\lceil\frac{L}{\varepsilon T}\right\rceil \alpha \\
& \leq & \varepsilon \left(\frac{L}{\varepsilon T} +1\right) \alpha \leq \left( \frac{L}{T}+\varepsilon_0\right) \alpha.
\end{eqnarray*}
Therefore, we have
\begin{align*}
	\left\|x^{\varepsilon}(t) - y^{\varepsilon}(t) \right\| & \leq \varepsilon K'' \int_{0}^{t} \sup_{u \in [0,s]} \left\|y^{\varepsilon}(u) - x^{\varepsilon}(u)\right\| {\rm d}h(s)+\varepsilon K + \varepsilon M \left( \frac{L}{T}+\varepsilon_0\right) \alpha,
	\end{align*}
where $K'' = (C + C_2 C_4) K'$. Let $\psi(s) = \displaystyle\sup_{\tau \in [0,s] } \left\|x^{\varepsilon}(\tau) - y^{\varepsilon}(\tau) \right\|$, then
		$$\psi(t) \leq \varepsilon K'' \int_{0}^{t} \psi(s) {\rm d}h(s)+\varepsilon K + \varepsilon M \left( \frac{L}{T}+\varepsilon_0\right) \alpha.$$
By Gronwall's inequality (Theorem \ref{Gronwall}), we get
		\begin{align*}
	\psi(t) & \leq e^{\varepsilon K''(h(t)-h(0))} \left(K + M \left( \frac{L}{T}+\varepsilon_0\right) \alpha\right) \cdot \varepsilon \\
	& \leq e^{K''\left( \frac{L}{T}+\varepsilon_0\right) \alpha} \left(K + M \left( \frac{L}{T}+\varepsilon_0\right) \alpha\right) \cdot \varepsilon.
	\end{align*}
		If we define $J := e^{K''\left( \frac{L}{T}+\varepsilon_0\right) \alpha} \left(K + M \left( \frac{L}{T}+\varepsilon_0\right) \alpha\right)$, then we have $\psi(t) \leq J\varepsilon$ for every $\varepsilon \in (0,\varepsilon_0]$ and $t \in [0,L/\varepsilon]$. Therefore, in particular,
	$$\left\|x^{\varepsilon}(t) - y^{\varepsilon}(t) \right\|_X = \sup_{t \in [0, L/\varepsilon]} \|x^{\varepsilon}(t)-y^{\varepsilon}(t)\|= \psi(L/\varepsilon) \leq J\varepsilon$$
 proving the desired result. \end{proof}

\begin{remark}
While our analysis currently assumes the same function $h$ across the integrals
of functions $f$ and $g$, it is feasible to employ distinct functions $h$ to derive more
general results. Moreover, given that measure equations can encompass dynamic equations
on various time scales, the use of two different functions $h$ could correspond to
different time scales. This flexibility allows for one $h$ to be associated with one time scale,
and another $h$ with a distinct time scale, thereby broadening the applicability and depth
of our results.
\end{remark}

\begin{remark}
By leveraging the correspondence between measure functional differential
equations with state-dependent delays and impulsive functional differential equations with
state-dependent delays, we can extend our findings to the latter. This enables us to establish
a version of the periodic averaging principle for these types of equations, which are
of significant practical importance. Moreover, as previously mentioned, these equations
also relate to functional dynamic equations on time scales. Utilizing this relationship, we
can apply the periodic averaging principle to these equations as well, thus broadening the
scope and impact of our results.
\end{remark}

\section{Application to extremum seeking through state-dependent delays via predictor feedback}

The work on predictor feedback has revealed applications in which the input or output delay is not constant. Among the most well-studied of such applications are additive manufacturing, shock waves in traffic flows, cooling systems, and internet congestion. While there are applications in which the delay varies over time in an open-loop fashion, as a result of drift, aging, and changes in operating conditions, arguably more interesting, and perhaps even prevalent, are situations where delay changes with time because it depends on the state of the system, which itself is time-varying. 

Since we have already proved that predictor feedback is capable of compensating delays in extremum seeking (ES) problems \cite{39}, it is natural to consider the compensation of delays that are nonconstant in ES problems. Predictors for stabilization of systems under delays that are state-dependent have already been introduced in \cite{BK:2013}. Several challenges arise in constructing predictors and in the resulting stability analysis when delays are state-dependent. One of the challenges is that, while the delay is known at each time instant, the prediction horizon, namely, the length of time it will take the input signal to reach the plant, is not a priori known at that time instant because that length of time is the inverse function of the difference between the current time and the delay. That time difference is called the ``delayed time.'' When the delay is an open-loop function, the prediction horizon can be found, in principle, as the inverse function of the delayed time. But, when the delay is state-dependent, it is not known, a priori, at what future time the input signal will reach the plant. It, therefore, takes certain transformations of the time variable and a reformulation of the integral equation for the predictor state to determine the prediction horizon and to predict the state value when the current input reaches the plant. 

This convoluted dynamic scenario, which bedevils the design process, has its manifestation in the analysis as well. If the delay is growing, and if it is growing too fast, the input signal being generated at the present instant may never reach the plant. It is, therefore, necessary that a rapid growth of the delay be preempted either by a priori assumption, in the case of open-loop nonconstant delays, or by design and restriction of the initial condition of the plant, in the particular case of state-dependent delays. In either case, to be specific, the delay rate (in the direction of growth) must not exceed unity. 
All of these considerations, which arise in predictor-based stabilization, carry over to extremum seeking. The said challenges have to be faced, dealt with, and can be overcome. This is what we do in this section. 

We propose an extremum seeking scheme for locally quadratic static maps in the presence of state-dependent delays. A predictor design using perturbation-based estimates of the unknown Gradient and Hessian of the map must be introduced to handle this variable nature of the delays, which can arise both in the input and output channels of the nonlinear map to be optimized. The demodulation signals commonly employed in ES must incorporate the state-dependent delays as well.\\

\subsection{Problem statement}

Scalar ES addresses the real-time control problem in which the aim is to find the 
maximum (or minimum) of the output $y\!\in\!\mathbb{R}$ of an unknown nonlinear static map $Q(\theta)$ by modifying the input $\theta\!\in\!\mathbb{R}$.

In this section, we assume there is an arbitrarily large \textit{state-dependent and known}
delay $D(\theta)\geq0$ in the actuation path or measurement system such that the map's 
output can be denoted by
\begin{eqnarray} \label{ch6.plant_state_dependent_delay}
y(t)&=&Q(\theta(t-D(\theta(t))))\,.
\end{eqnarray}
The delay $D(\theta)$ is a nonnegative-valued continuously differentiable function.  
For the sake of clarity, we assume here that the map is output-delayed, as depicted in the block diagram in Figure~\ref{ch6.Fig1}.
However, the results provided 
here can be directly
generalized to the input-delay case once any input time-delay can be moved to the output
of the static map. The general setup when input delays $D_{\rm{in}}(\theta(t))$ and output delays $D_{\rm{out}}(\theta(t))$ occur
concurrently can also be handled, by assuming that the total delay to be compensated is 
$D(\theta(t))=D_{\rm{in}}(\theta(t))+D_{\rm{out}}(\theta(t))$, with $D_{\rm{in}}(\theta(t))\,,D_{\rm{out}}(\theta(t))\geq 0$.
%
%
%
\begin{figure*}[t]
\begin{center}
\vspace{-0.5cm}
\includegraphics[scale = 0.3]{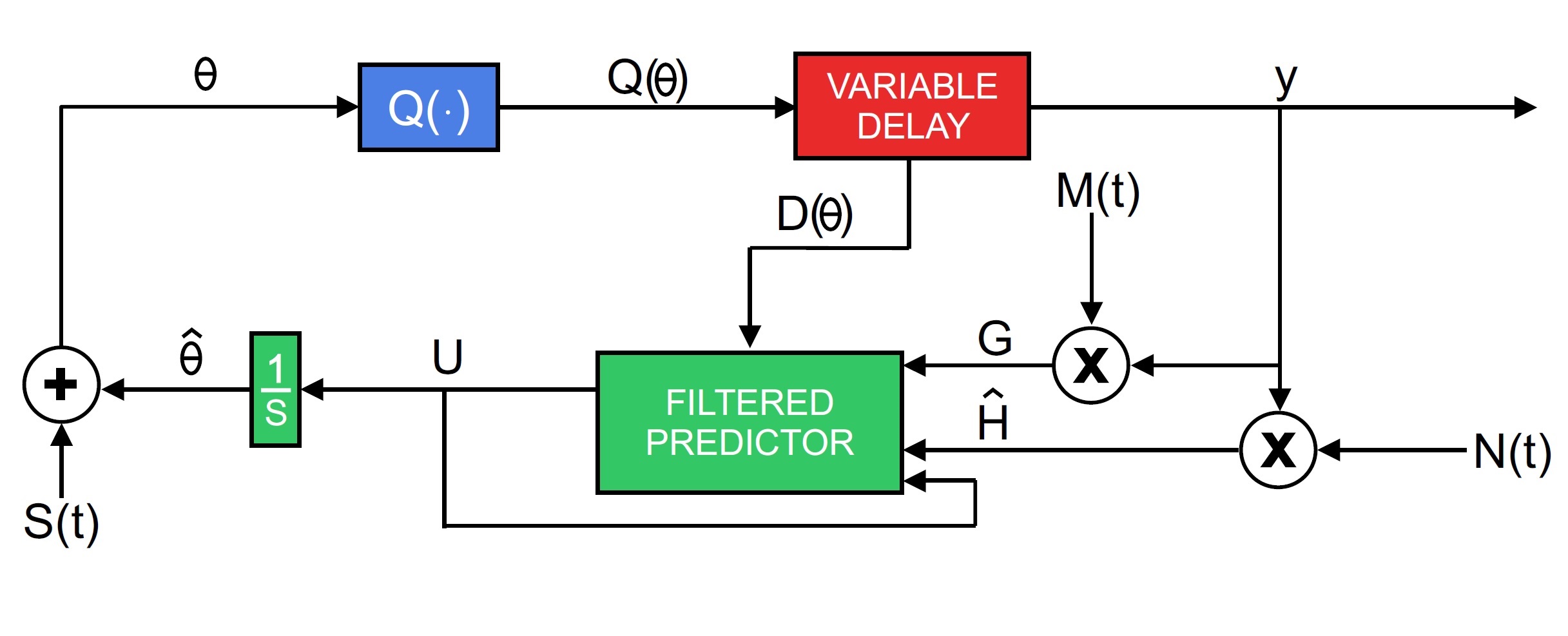}
\caption{Block diagram of the extremum seeking control system with nonconstant delays.}\label{ch6.Fig1}
\end{center}
\end{figure*}
%

Without loss of generality, let us consider the maximum seeking
problem such that the maximizing value of $\theta$ is denoted by
$\theta^*$. For the sake of simplicity, we also assume that the
nonlinear map is at least locally quadratic: 
\begin{eqnarray} \label{ch6.nonlinear_map}
Q(\theta)&=&y^*+\frac{H}{2}(\theta-\theta^*)^2\,,
\end{eqnarray}
where besides the constants $\theta^*\!\in\!\mathbb{R}$ and $y^*\!\in\!\mathbb{R}$ being unknown,
the scalar $H\!<\!0$ is the unknown Hessian of the static map.
By plugging (\ref{ch6.nonlinear_map}) into (\ref{ch6.plant_state_dependent_delay}), we obtain the \textit{quadratic static map with delay} of interest:
\begin{eqnarray} \label{ch6.delayed_output}
y(t)&=&y^*+\frac{H}{2}(\theta(t-D(\theta(t)))-\theta^*)^2\,.
\end{eqnarray}
This problem is particularly important in distributed extremum seeking control of mobile robots or formation control with multiple agents, where the magnitude of the delay may depend on the distance of the robots/agents from the operator interface, or from the virtual leaders \cite{34},\cite{47}. 

One more interesting feature is that the delay function can even depend on unmeasured signals, but respecting the condition that at least the analytical expression (parametrization) for the delays is still known and the unmeasured signals can be estimated in the average sense. Without loss of generality, let us assume here that the delay function in (\ref{ch6.plant_state_dependent_delay}) is:  
\begin{equation} \label{ch6.state_dependent_delay}
D(\theta(t))=Q'(\theta(t))=H(\theta(t)-\theta^*)\,,
\end{equation}  
with $Q(\theta)$ defined in (\ref{ch6.nonlinear_map}). In this case, the delay function is exactly the expression of the gradient of the map we want to maximize.   

Considering the continuously differentiable functions of the \textit{delay time} $\phi(t)$ and \textit{prediction time} $\phi^{-1}(t)$, defined as
\begin{align} \label{ch6.eq19}
\phi(t) &:= t - D(\theta(t))\,, 
\end{align}
%
where $\phi^{-1}(t)$ is the inverse function of $\phi(t)$. We assume the invertibility of $\phi(t)$ along the section. When the delay is constant, $\phi^{-1}(t)\!=\!t\!+\!D$. The following assumption is also made. 


\begin{assum} \label{Assumption 1_ch.6}
The output time-varying delay $D(\theta(t)) $ fulfills
\begin{eqnarray}
-\infty < \underline{d} < \dot{D}(\theta(t)) < \overline{d} < 1,
\label{ch6.eq23}
\end{eqnarray}
where $\underline{d}$ and $\overline{d}$ are constants. 
\end{assum}

This is a natural consequence of (\ref{ch6.eq19}). 
In addition, from Assumption~\ref{Assumption 1_ch.6} it follows that $\phi(t)$ is strictly increasing, \textit{i.e.}, 
\begin{eqnarray}
\phi(t) \leq t\,,  \qquad  \forall t \geq 0.
\label{ch6.eq21}
\end{eqnarray}
Inequality (\ref{ch6.eq21}) can be alternatively states as 
\begin{eqnarray}
\phi^{-1}(t)-t >0.
\label{ch6.eq21_enrolation}
\end{eqnarray}

\subsection{Probing and demodulation signals}
\label{ch6.sec:signals}

Let $\hat{\theta}$ be the estimate of $\theta^*$ and
\begin{eqnarray}\label{ch6.estimation_error}
\tilde{\theta }(t) = \hat{\theta}(t) - \theta ^{*}
\end{eqnarray}
be the \textit{estimation error}. From Figure~\ref{ch6.Fig1} and noting that $\dot{\hat{\theta}}(t)=\dot{\tilde{\theta}}(t)$, the \textit{error dynamics} can be written as
\begin{eqnarray}
\dot{\tilde{\theta }}(t-D(\theta(t))) = U(t-D(\theta(t)))\,, \label{ch6.distinct_delayed_equation}
\end{eqnarray}
where $U(t)$ is the control signal. Additionally, we have 
\begin{eqnarray}
G(t) = M(t)y(t), \quad \theta (t) = \hat{\theta }(t) + S(t), \label{ch6.G}
\end{eqnarray}
with 
the probing and demodulation signals being given by: 
\begin{eqnarray}
S(t) = a\sin (\omega t), \quad M(t)=\frac{2}{a}\sin (\omega (t-D(\theta(t)))), \label{ch6.dither_signalS} 
\end{eqnarray}
with nonzero perturbation amplitude $a$ and frequency
$\omega$.
%
%
The signal
\begin{equation}\label{ch6.hessian_estimate_H}
\hat{H}(t) = N(t)y(t)
\end{equation}
is applied to obtain an estimate of the unknown Hessian $H$, where
the demodulating signal $N(t)$ is given by
\begin{equation}\label{ch6.hessian_estimate_N}
N(t) = -\frac{8}{a^{2}}\cos (2\omega (t-D(t))).
\end{equation}
In \cite{GKN:2012}, it was proved that
\begin{equation}\label{ch6.averagingH}
\frac{1}{T}\int_{0}^{T }N(\sigma )yd\sigma = H, \quad T = 2\pi/\omega,
\end{equation}
if a quadratic map as in (\ref{ch6.nonlinear_map}) is considered. In
other words, it's average version is $\hat H_{\rm{av}}=(Ny)_{\rm{av}}=H$.

\subsection{Predictor feedback for state-dependent delays} \label{predictor_design_gold}

Now, we conveniently redefine the estimation error introduced in (\ref{ch6.estimation_error}) as
\begin{equation} \label{ch6.estimation_error_state_dependent_delay}
\tilde{\theta}(t)=\hat{\theta}(t-D(\theta(t)))-\theta^*\,,
\end{equation}  
and recalling that $\dot{\tilde{\theta}}(t)=\dot{\hat{\theta}}(t-D(\theta(t)))=U(t-D(\theta(t)))$, we can write the associated error dynamics as
\begin{equation} \label{ch6.error_dynamics_state_dependent_delay}
\dot{\tilde{\theta}}(t)=U(t-D(\theta(t)))\,,
\end{equation}  
where $U(t)$ is the control signal. 

Similarly to \cite{40}, but now assuming (\ref{ch6.plant_state_dependent_delay}) and (\ref{ch6.estimation_error_state_dependent_delay}), the following average version of the signal $G(t)$ in (\ref{ch6.G}) can be derived straightforwardly by  
\begin{equation} \label{ch6.Gav_state_dependent_delay}
G_{\rm{av}}(t)=H\tilde{\theta}_{\rm{av}}(t)\,,
\end{equation}

From (\ref{ch6.error_dynamics_state_dependent_delay}) and (\ref{ch6.Gav_state_dependent_delay}), the following average models can also be obtained
\newpage
\begin{align}
\dot{\tilde{\theta}}_{\rm{av}}(t)&=U_{\rm{av}}(\phi(t))\,, \label{ch6.thetadot-state}\\
\dot G_{\rm{av}}(t)&= H U_{\rm{av}}(\phi(t))\,, \label{ch6.Gdot-state}\\
 \phi(t)&=t-D(G_{\rm{av}}(t))\,.\label{ch6.delay-state}
\end{align}
since $\theta(t-D(\theta(t)))=\hat{\theta}(t-D(\theta(t)))+S(t-D(\theta(t)))= \tilde{\theta}(t)+ \theta^* + a \sin(\omega (t-D(\theta(t))))$ from the definitions in (\ref{ch6.G}) and (\ref{ch6.estimation_error_state_dependent_delay}). Hence, the average signal $[\theta(t-D(\theta(t)))-\theta^*]_{\rm{av}}=[\tilde{\theta}(t)]_{\rm{av}}$ since $[a \sin(\omega (t-D(\theta(t))))]_{\rm{av}}=0$. By multiplying both sides of the former equation by $H$, from (\ref{ch6.state_dependent_delay}) and (\ref{ch6.Gav_state_dependent_delay}), we can write, with some abuse of notation,
$[D(\theta(t))]_{\rm{av}}\!=\!G_{\rm{av}}(t)\!:=\! D(G_{\rm{av}}(t))$. The term
$U_{\rm{av}}: [\phi(t_0)\,, \infty)\!\to\! \mathbb{R}\,, t \!\geq\! t_0 \!\geq\! 0\,,
D \!\in\! C^{1}(\mathbb{R}; ~\mathbb{R}_{+})$ is the average control of $U \!\in\! \mathbb{R}$.

If we simply apply a proportional gradient control law $U(t)=KG(t)$ to (\ref{ch6.Gdot-state}), with $k>0$ being a positive constant, we could write the following expression
\begin{equation} \label{Mesquita1}
\dot{G}_{\rm{av}}(t)=HkG_{\rm{av}}(t-D({G}_{\rm{av}}(t)))\,.
\end{equation}  
From (\ref{Mesquita1}),  it is clear that the equilibrium of the average
system is not necessarily stable for arbitrary values of the
delay $D({G}_{\rm{av}}(t))$. This reinforces the necessity of applying the prediction to stabilize the system.

In this sense, if we implement a proportional control law of the implementable future state $U(t)=kG(\phi^{-1}(t))$, the average closed-loop system would become,
\begin{equation} \label{Mesquita2}
\dot{G}_{\rm{av}}(t)=HkG_{\rm{av}}(t)\,,
\end{equation}  
which has an exponentially stable equilibrium $G_{\rm{av}}^{\rm{e}}(t)\equiv0$ since $Hk<0$. 

The main challenge in the case of systems with state-dependent delays is the determination of the predictor state. For
systems with constant delays, $D = \text{const}$, the predictor of the state
$G(t)$ is simply defined as $P(t) = G(t+D)$. For systems with state-dependent delays finding the predictor $P(t)$ is much trickier. The
time when $U$ reaches the system depends on the value of the state at that time, namely, the following implicit relationship holds
$P(t) = G(t + D(P(t)))$ (and $G(t) = P(t-D(G(t)))$).

Hence, by redefining the prediction time as $\sigma(t)=\phi^{-1}(t)$ the inversion of the time variable $t \to t-D(G_{\rm{av}}(t))$ in
$t\!\to\!t+D(P_{\rm{av}}(t))$, the following hold:
\begin{align}
P_{\rm{av}}(t)&=G_{\rm{av}}(\sigma(t))\,, \label{ch6.pav-state}\\
 \sigma(t)&=t+D(P_{\rm{av}}(t))\,,\label{ch6.sigma-state}
\end{align}
where $P_{\rm{av}}(t)$ is the associated predictor state. 

Differentiating (\ref{ch6.pav-state}) and (\ref{ch6.sigma-state}) and using (\ref{ch6.Gdot-state}), we arrive at \cite{BK:2013_b}: 
\begin{equation} \label{ch6.bucetasso1}
\frac{dG_{\rm{av}}(\sigma(t))}{dt}=H U_{\rm{av}}(t) \frac{d\sigma(t)}{dt}\,,
\end{equation}
and
\begin{equation} \label{ch6.bucetasso2}
\dot{\sigma}(t)= \frac{1}{1-H\nabla D(P_{\rm{av}}(t))U_{\rm{av}}(t)}\,,
\end{equation}
with $\dot{D}\!=\!\nabla D \dot{G}_{\rm{av}}\!\!=\!\!\nabla D HU_{\rm{av}}$, $\nabla$ being the gradient operator. 

The following implicit integral relations \cite{BK:2013_b} are derived for the predictor state and the prediction time by integrating
(\ref{ch6.bucetasso1}) and (\ref{ch6.bucetasso2}) for all delay interval $\phi(t)\leq\Theta\leq t$:
\begin{align}
P_{\rm{av}}(\Theta)&=G_{\rm{av}}(t) + \int_{\phi(t)}^{\Theta} 
\frac{HU_{\rm{av}}(s)}{1-H\nabla D(P_{\rm{av}}(s))U_{\rm{av}}(s)}ds\,, \label{ch6.pav-state_cu}\\
 \sigma(\Theta)&=t+ \int_{\phi(t)}^{\Theta} \frac{1}{1-H\nabla D(P_{\rm{av}}(s))U_{\rm{av}}(s)}ds \,.\label{ch6.sigma-state_cu}
\end{align}
Regarding the nominal feedback control gain $k>0$ for the plant free of delays, the predictor feedback control law can be written as  
\begin{align}
U_{\rm{av}}(t)\!&=\! k\left[G_{\rm{av}}(t) \!+\! \int_{\phi(t)}^{t} 
\frac{HU_{\rm{av}}(s)}{1-H\nabla D(P_{\rm{av}}(s))U_{\rm{av}}(s)}ds\right]. \label{ch6.control_law_avg_cu}
\end{align}
From Assumption~\ref{Assumption 1_ch.6}, the predictor feedback control law (\ref{ch6.control_law_avg_cu}) is subject to the following feasibility condition \cite{BK:2013_b}: 
\begin{equation} \label{ch6.feasibility_condition}
H\nabla D(P_{\rm{av}}(t))U_{\rm{av}}(t) < 1\,.
\end{equation}
This is due to the fact that the prediction signal can be only generated if (\ref{ch6.pav-state_cu}) and (\ref{ch6.sigma-state_cu}) are
well-posed. If the condition (\ref{ch6.feasibility_condition}) is violated, the control signal is directed toward the opposite direction to the plant. 

Thus, the non-average version for the filtered predictor-based control law is given by
\begin{align} \label{ch6.control_state_dependent1}
U(t)&=\frac{c}{s+c}\left \{k \left[ G(t) \!+\! \Gamma(t)  \right]    \right\}\,, \quad k\!>\!0 \,, \quad c\!>\!0 \,, \\
\Gamma(t) &= \hat{H}(t) \int_{\phi(t)}^{t} 
\frac{U(\tau)}{1-\hat{H}(\tau)\nabla D(G(\tau))U(\tau)}d\tau\,. \label{ch6.control_state_dependent2gold}
\end{align}

The stability proof can be established using an analogous average infinite-dimensional representation of the actuator state \cite{40}, where the infinite dimensional dynamics of the delay is represented by a transport partial differential equation (PDE) with the propagation speed defined as 
\begin{equation} \label{speedpropagationmesquita}
\pi(x,t)=\frac{1+x(\dot{\sigma}(t)-1)}{\sigma(t)-1}\,.
\end{equation}
The stability analysis for the closed-loop system is carried out in the next section. 

\subsection{Stability analysis and convergence to the extremum}

The infinite-dimensional representation for the system (\ref{ch6.thetadot-state})--(\ref{ch6.delay-state}) is 
\begin{eqnarray}
%
\dot{\tilde{\theta }}_{\rm{av}}\left ( t \right ) &=& u _{\rm{av}}\left ( 0,t \right ),
\label{ch6.eq57_state_dependent}\\
%
\partial_{t}u_{\rm{av}}\left ( x,t \right ) &=& \pi(x,t)\partial_{x}u_{\rm{av}}\left ( x,t \right ),  ~~x \in (0\,,1),
\label{ch6.eq58_state_dependent}\\
%
\frac{du_{\rm{av}}\left ( 1,t \right )}{dt} &=& - c u_{\rm{av}}(1,t) + ckH\left[ \tilde{\theta}_{\rm{av}} \!+\! (\sigma(t)\!-\! t) \!\! \int_{0}^{1} \!\!\!u_{\rm{av}}(x,t) dx  \right]\,,~~~~~
\label{ch6.eq59_state_dependent}
%
\end{eqnarray}
which can be mapped into the \textit{target system} \cite{BK:2013_b}: 
\begin{eqnarray}\label{ch6.eq65_state_dependent}
\dot{\tilde{\theta}}_{\textrm{av}}(t) &=& kH \tilde{\theta}_{\rm{av}}(t) + w(0,t)\,, \\
w_{t}(x,t) &=& \pi(x,t) w_{x}(x,t), \quad x \in (0, 1)\,, \label{ch6.eq66_state_dependent} \\
w(1,t) &=& -\frac{1}{c}\partial_{t}u_{\rm{av}}(1,t)\,, \label{ch6.eq67_state_dependent}
\end{eqnarray}
with the help of the following invertible backstepping transformation:
\begin{align}
w(x,t) &= u_{\textrm{av}}(x,t) - kH \left[ \tilde{\theta}_{\textrm{av}}(t) + (\sigma(t)-t)\int_{0}^{x} u_{\textrm{av}}(\sigma,t) d\sigma \right]\,.
\label{ch6.eq64_state_dependent}
\end{align}
Considering the next Lyapunov-Krasovskii functional 
\begin{equation}\label{ch6.Lyapunov_state_dependent}
V(t)\!=\!\frac{{\tilde\theta}^2_{\rm{av}}(t)}{2}+\frac{a}{2}\int_{0}^{1}e^{bx}w^2(x,t) dx + \frac{1}{2}w^2(1,t)\,,
\end{equation}
with appropriate constants $a>0$ and $b>0$, the proof of stability can be provided following analogously the same procedure for constant delays \cite{40}. 
The main difference is that the delay time $\phi(t)$ depends on the state as well as the prediction time
$\phi^{-1}(t)$---here denoted by $\sigma(t)$. Moreover, the infinite-dimensional averaging theorem to be invoked here is the proposed Theorem~\ref{3} of Section~\ref{teoremaprincipalmesquita} rather than
the Hale and Lunel averaging theorem in \cite{18}. The latter one is limited to constant delays.  

Using (\ref{ch6.Lyapunov_state_dependent}), one can show that there exists a sufficiently small $\lambda>0$ such that 
\begin{equation}
\dot{V} \leq -\lambda V\,.
\end{equation}
From this result, along with 
\begin{equation} \label{medvedev}
m_1 \Psi(t) \leq V(t) \leq m_2 \Psi(t)
\end{equation}
where $\Psi(t)=\tilde{\theta}_{\rm{av}}^2(t) + \|w\|^2 + w(1,t)^2$, it follows that 
\begin{equation}
\dot{\Psi}(t) \leq M e^{-t/M} \Psi(0)\,,
\end{equation}
for a sufficiently large $M>0$. In (\ref{medvedev}), we denote the spatial $\mathcal{L}_{2}[0,1]$ norm of the PDE state $w(x, t)$ as $\|w(t)\|^{2} \coloneqq \int_{0}^{1} w^2(x,t) \,dx$.

From the invertibility of the backstepping transformation (\ref{ch6.eq64_state_dependent}) \cite{28}, it follows that
\begin{align} \label{Alcaraz}
& {} |\tilde{\theta}_{\rm{av}}(t)|^{2}\!+\!\|u(t)\|^{2}\!+\!|u(1,t)|^{2} \\ \nonumber
& {}\leq  \bar{M} e^{-t/\bar{M}}  (|\tilde{\theta}_{\rm{av}}(0)|^{2}\!+\!\|u_{\rm{av}}(0)\|^{2}\!+\!|u_{\rm{av}}(1,0)|^{2})\,, \quad \bar{M}>0\,, \quad \forall t \!\geq\! 0 \,. 
\end{align}
%

From (\ref{Alcaraz}), the origin of the average closed-loop system with transport PDE and nonconstant speed propagation is exponentially stable. Then, according to the proposed Averaging Theorem~\ref{3} of Section~\ref{teoremaprincipalmesquita}, for $\omega$ sufficiently large, the closed-loop system (\ref{ch6.thetadot-state})--(\ref{ch6.delay-state}), with $U(t)$ in (\ref{ch6.control_state_dependent1}) and (\ref{ch6.control_state_dependent2gold}), has a unique exponentially stable periodic solution around its equilibrium (origin) satisfying:
\begin{equation}
\begin{aligned}
& {} \left(|\tilde{\theta}^{\Pi}(t)|^{2}+\|u^{\Pi}(t)\|^{2}+|u^{\Pi}(1,t)|^{2} \right)^{1/2} \leq \mathcal{O}(1/\omega)\,, \quad \forall t \geq 0\,,
\label{extrachapter.eq:order_period}
\end{aligned}
\end{equation}
 $\vartheta^{\Pi}(t)$, $u^{\Pi}(x,t)$ represents the unique locally exponentially stable periodic solution in $t$ with a period $\Pi \coloneqq 2\pi/\omega$. In (\ref{extrachapter.eq:order_period}), the big-$\mathcal{O}$ notation\footnote{As defined in \cite{27}, a vector function $f(t,\epsilon) \in \mathbb{R}^{n}$ is said to be of order $\mathcal{O}(\epsilon)$ over an interval $[t_{1},t_{2}]$, if $\exists k, \bar{\epsilon}$ : $|f(t,\epsilon)| \leq k\epsilon , \forall t \in [t_{1},t_{2}]$ and $\forall \epsilon \in [0,\overline{\epsilon}]$. In general, we do not provide precise estimates for the constants $k$ and $\bar{\epsilon}$, and we use $\mathcal{O}(\epsilon)$ to be interpreted as an order-of-magnitude relation for sufficiently small $\epsilon$.} is used to characterize the ultimate residual set.

On the other hand, the asymptotic convergence to a neighborhood of the extremum point is proved taking the absolute value of the second expression in (\ref{ch6.G}), with $S(t)$ defined in (\ref{ch6.dither_signalS}), after replacing 
$\hat{\theta}(t-D(\theta(t)))=\tilde{\theta}(t) + \theta^{*}$ from (\ref{ch6.estimation_error_state_dependent_delay}), resulting in:
\begin{equation}
    |\theta(t)-\theta^{*}| = |\tilde{\theta}(t) + a\sin{(\omega t)}|. \label{extrachapter.vasin2}
\end{equation} 
Considering  \eqref{extrachapter.vasin2} and writing it by adding and subtracting the periodic solution $\tilde{\theta}^\Pi(t)$, it follows
\begin{equation}
    |\theta(t)-\theta^{*}| = |\tilde{\theta}(t)-\tilde{\theta}^\Pi(t)+\tilde{\theta}^\Pi(t)  + a\sin{(\omega t)}|. \label{extrachapter.vasin2_FONSECA}
\end{equation} 
By applying the proposed Averaging Theorem~\ref{3} of Section~\ref{teoremaprincipalmesquita}, one can conclude that  $\tilde{\theta}(t)\!-\!\tilde{\theta}^\Pi(t)\!\to\!0$ exponentially. Consequently, 
\begin{equation}
    \limsup_{t \to \infty} |\theta(t)\!-\!\theta^{*}| = \limsup_{t \to \infty} |\tilde{\theta}^{\Pi}(t) + a\sin{(\omega t)}| \label{extrachapter.Order_Theta(t)1}.
\end{equation}
Finally, utilizing the relationship (\ref{extrachapter.eq:order_period}), we ultimately arrive at
\begin{align}
\limsup_{t \to \infty} |\theta(t)-\theta^{*}| = \mathcal{O}\left(a+1/\omega\right)\,. \label{Order_Theta(t)}
\end{align}
%


In order to show the convergence of the output $y(t)$, we can follow the same steps employed for $\theta(t)$ by plugging (\ref{Order_Theta(t)}) into  (\ref{ch6.delayed_output}), such that
\begin{equation}
\begin{split}
\limsup_{t \to \infty} |y(t)-y^{*}| = \limsup_{t \to \infty} |H\tilde{\theta}^{2}(t) + Ha^{2}\sin{(\omega t)}^{2}|. \label{extrachapter.Order_y(t)1}
\end{split}
\end{equation}
Hence, by rewriting (\ref{extrachapter.Order_y(t)1}) in terms of $\tilde{\theta}^{\Pi}(t)$ and again with the help of (\ref{extrachapter.eq:order_period}), we finally get
\begin{align}
\limsup_{t \to \infty} |y(t)-y^{*}| = \mathcal{O}\left(a^{2}+1/\omega^{2}\right). \label{order_y(t)}
\end{align}
Hence, from (\ref{Order_Theta(t)}) and (\ref{order_y(t)}) we rigorously conclude the convergence of $\theta(t)$ and $y(t)$ to a small neighborhood of the maximizer $\theta^*$ and the extremum point $y^*$, respectively. In the next section, we illustrate this theoretical result through a numerical academic example. 

\subsection{Simulation example}

The numerical simulation considers the quadratic map described in (\ref{ch6.delayed_output}), with simulation parameters selected according to Table \ref{tb:margins}. In particular,
we have assumed the same delay suggested in \cite[Example~15.3, page 252]{BK:2013} equal to $D(\theta(t))=\frac{1}{2}\sin(5\theta(t))^2$. Despite its small amplitude, such a nonconstant delay is enough to destabilize the closed-loop system if not properly compensated with the predictor presented in Section~\ref{predictor_design_gold}. In this case, the delay time $\phi(t)$ and the prediction time $\phi^{-1}(t)$ are illustrated in
Figure~\ref{delay_prediction_times}. 

As shown in \cite[Part III]{BK:2013}, state-dependent delays can be described by a transport PDE containing a nonconstant speed propagation
$\pi(x,t)$ in (\ref{speedpropagationmesquita}) given by 
\begin{align}
&    \Theta(t) = \alpha(0,t)=\theta(t-D(\theta(t)))\,,\label{Theta_p_actuator}\\
&    \partial_{t}\alpha(x,t) = \pi(x,t) \partial_{x}\alpha(x,t)\,, \quad x \in (0\,,1)\,, \label{wavekv_actuator}\\
&    \alpha(1,t) = \theta(t)\,. \label{theta_actuator}
\end{align}
This is exactly as we generate the \textcolor{black}{red block} ``Variable Delay'' in Figure~\ref{ch6.Fig1} where $\alpha:[0,1]\times\mathbb{R}_{+}\rightarrow\mathbb{R}$ and
$\alpha(x,t) = \theta(\phi(t + x(\phi^{-1}(t) - t)))$. 
%

\begin{small}
\begin{table}[htb!] 
\begin{center}
\caption{Simulation parameters} \label{tb:margins}
\begin{tabular}{cccc}
 & Symbol & Description & Value \\\hline
                      & $K$ & controller gain & 0.2 \\
    Controller        & $c$ & controller frequency [rad/s] & 2\\
    parameters        & $a$ & pertubation amplitude & 0.2\\
                      & $\omega$ & pertubation frequency [rad/s] & 8\\
                      
     \hline
                      & $\theta^{*}$ & optimizer static map & 8 \\
    System            & $y^{*}$ & optimal value static map & 64 \\
    parameters        & $H$ & Hessian  & -1 \\
                      & $D(\theta(t))$ & state-dependent delay & $D(\theta(t))=\frac{1}{2}\sin(5\theta(t))^2$ \\
\end{tabular}
\end{center}
\end{table}
\end{small}
\begin{figure} [htb!]
\begin{center}
\includegraphics[width=7.4cm]{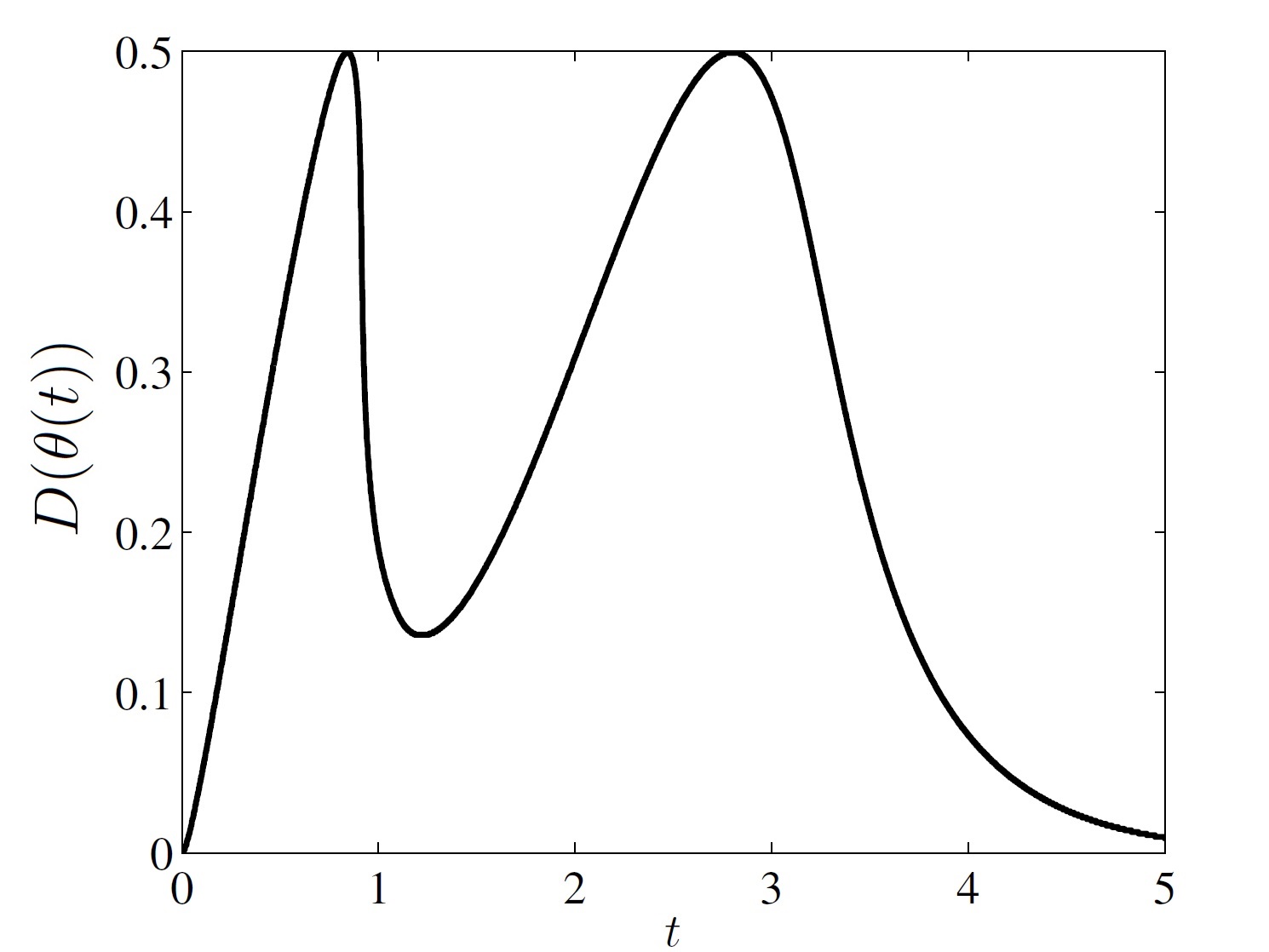}
\includegraphics[width=7.6cm]{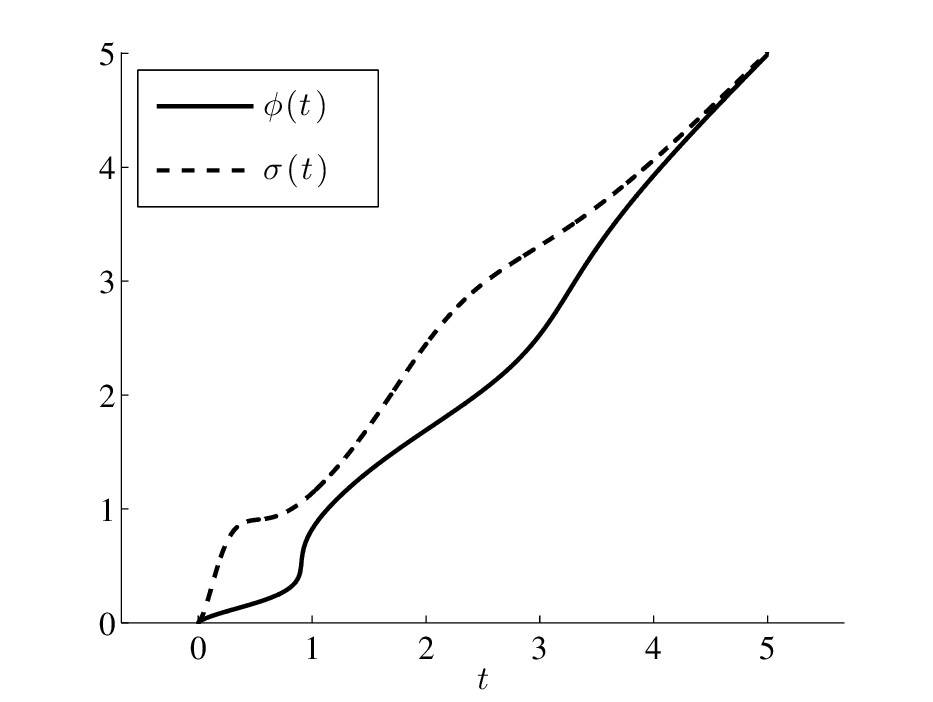} 
\caption{The state-dependent-delay $D(\theta(t))=\frac{1}{2}\sin(5\theta(t))^2$, the delayed time $\phi(t)=t-\frac{1}{2}\sin(5\theta(t))^2$, and the prediction time $\sigma(t)=\phi^{-1}(t)$.} 
\label{delay_prediction_times}
\end{center}
\end{figure}
%

Figure \ref{fig:3dtheta}  
corresponds to the numerical plot of the closed-loop system evolution in a three-dimensional space, taking into account the domain $x \in [0,1]$ and the time $t$. The curves in blue and in red show the convergence of $\Theta(t)=\alpha(0,t)$ and $\theta(t)=\alpha(1,t)$ to a small neighborhood around the optimizer $\theta^{*} =8$, respectively.

%
%

%
%
\begin{figure*} [htb!]
\begin{center}
\includegraphics[width=16cm]{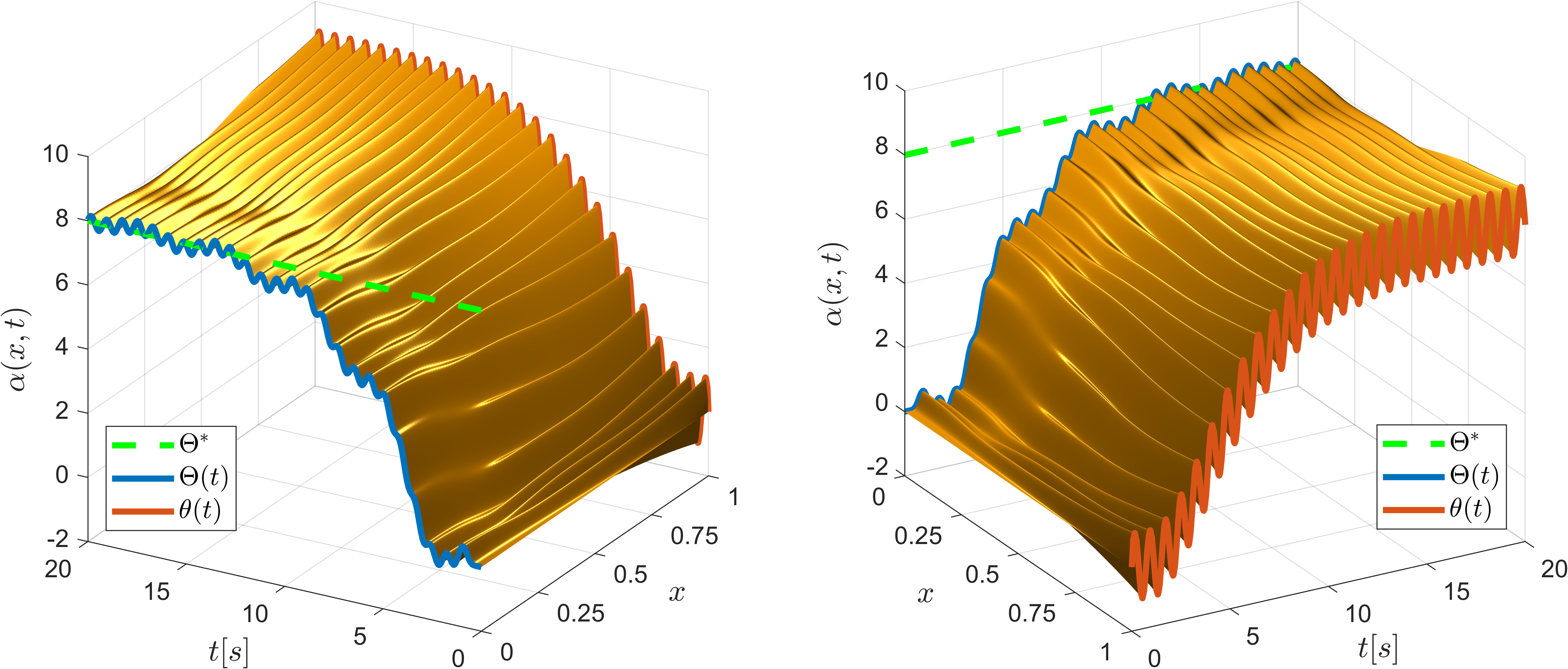} 
\caption{Convergence of the state $\alpha(x,t)$ employed to represent the delayed system with state-dependent delays \cite{BK:2013} in a three-dimensional space. In blue, we can check $\Theta(t)=\alpha(0,t)$, while in red, we have $\theta(t)=\alpha(1,t)$, both reaching a neighborhood of $\theta^\ast$.} 
\label{fig:3dtheta}
\end{center}
\end{figure*}
%
%
%

Finally, Figures \ref{fig:y_t} and \ref{fig:u_t} show the convergence of $y(t)$ and $U(t)$ to small neighborhoods of $y^{*}$ and $0$, as expected.
\begin{figure} [htb!]
\begin{center}
\includegraphics[width=11cm]{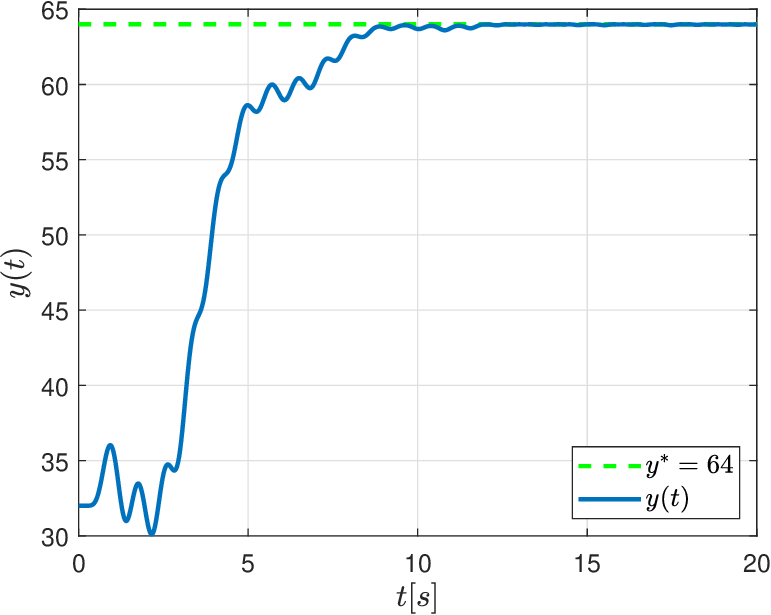}
\caption{Convergence of the output $y(t)$ to a neighborhood of $y^{*}$.} 
\label{fig:y_t}
\end{center}
\end{figure}
\begin{figure} [htb!]
\begin{center}
\includegraphics[width=11cm]{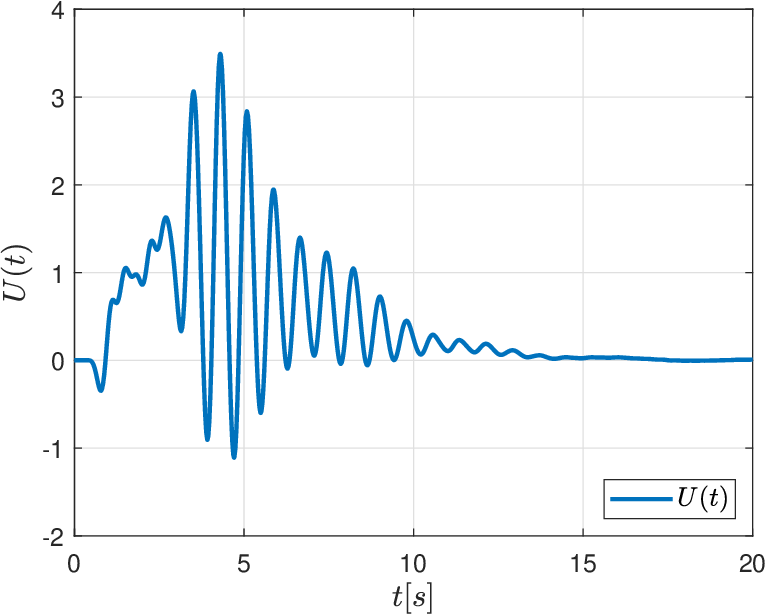}  
\caption{Convergence of the control signal $U(t)$ to a neighborhood of $0$.} 
\label{fig:u_t}
\end{center}
\end{figure}
%
%

\newpage
\section{Conclusions}

This paper embarked on a detailed exploration of the intricate dynamics of measure
functional differential equations with state-dependent delays. Our initial efforts centered
on establishing the existence and uniqueness of solutions for these equations, treating them
as fixed points of the solution operator. This foundational approach set the stage
for our pioneering investigation into a periodic averaging principle---a novel contribution
that has not been previously documented in the literature.

The core challenges associated with this type of problem stem from the regularity
of the functions involved and the complexities of the phase spaces utilized. Given that
equations involving measures may include discontinuous functions, and considering that
the solution itself acts as the state in the function of delay which is then recomposed
with the function, we encountered significant hurdles in deriving meaningful results. Despite
these challenges, the versatility of measure functional differential equations allows
them to be conceptualized as other types of equations, such as impulsive equations with
state-dependent delays and difference equations with state-dependent delays. This insight
extends the applicability of our findings, presenting new opportunities for research
in these areas. While the direct implications for these equations are not discussed in detail
here due to their straightforward nature and the potentially uninteresting proofs, they are
readily derivable by the interested reader from the foundational results presented in this
study.

Further advancing our research, we applied our theoretical developments to a specific
case involving a partial differential equation (PDE) and demonstrated their relevance to a
real-time optimization strategy known as extremum seeking. Notably, prior to our work,
there were no averaging theorems capable of analyzing the stability of extremum seeking
algorithms under the conditions of state-dependent delays. With these new tools at our
disposal, we successfully established the stability of a novel extremum seeking algorithm
that utilizes predictor feedback for static maps affected by state-dependent delays.

In conclusion, the contributions of this paper not only address a gap in the existing
mathematical literature but also offer practical tools for engineers and scientists working
with complex dynamical systems. The theoretical principles we have developed have
significant implications for the design and stability analysis of advanced control systems,
potentially guiding future innovations in various technical and scientific fields.

\section*{Acknowledgment}

The authors would like to dedicate this work to the memory of Professor Dr. Hernán R.
Henríquez, who passed away on June 2022, 2024. This manuscript was discussed with him before
his death and Professor Hernán gave many important advices for the improvement of this
paper. The authors are very grateful to Professor Henríquez for the fruitful discussions.
The authors have a lot of good memories of him, and will always remember him with
love, gratitude and respect.

\end{document}